\documentclass[onefignum,onetabnum]{siamart190516}

\usepackage[utf8]{inputenc}
\usepackage[english]{babel}
\usepackage[nolist]{acronym}
\usepackage{cite}
\usepackage{mathtools}
\usepackage{amsfonts}
\usepackage{amsmath}

\usepackage{hyperref}
\usepackage{pgfplots}
\usepackage{tikz}

\usepackage{algorithm}
\usepackage{algorithmic}

\hypersetup{
    colorlinks,
    linkcolor={red!50!black},
    citecolor={blue!50!black},
    urlcolor={blue!80!black}
}

\usepackage[scientific-notation=true]{siunitx}
\usepackage{graphicx}
\usepackage{todonotes}
\usepackage{subcaption}

\newsiamremark{remark}{Remark}
\newsiamremark{example}{Example}
\newsiamthm{assumption}{Assumption}

\newcommand*\dd{\mathop{}\!\mathrm{d}}

\DeclareMathOperator{\N}{\mathbb{N}}
\DeclareMathOperator{\spn}{span}
\DeclareMathOperator{\conv}{conv}
\DeclareMathOperator{\solve}{\texttt{SOLVE}}
\DeclareMathOperator{\podop}{\texttt{POD}}
\DeclareMathOperator{\PPi}{\Pi}

\newcommand{\cf}{cf.\ }
\newcommand{\etal}{et al.\ }

\newcommand{\ie}{i.e.\ }

 	 	 	\newcommand{\C}[1]{}
\newcommand{\R}{\mathbb{R}}

\pgfplotsset{select coords between index/.style 2 args={
    x filter/.code={
        \ifnum\coordindex<#1\fi
        \ifnum\coordindex>#2\fi
    }
}}

\headers{A simplified Newton method for POD snapshot generation}{P. Manns  and S. Ulbrich}

\author{Paul Manns\thanks{Mathematics and Computer Science Division, Argonne National Laboratory, Lemont, IL 
  (\email{pmanns@anl.gov}).}
\and Stefan Ulbrich\thanks{Department of Mathematics, TU Darmstadt, Darmstadt, Germany
  (\email{ulbrich@mathematik.tu-darmstadt.de}).}}

\ifpdf
\hypersetup{
  pdftitle={A simplified Newton method to generate snapshots for POD models of semilinear optimal control problems},
  pdfauthor={P. Manns and S. Ulbrich}
}
\fi

\title{A simplified Newton method to generate snapshots for POD models of semilinear optimal control problems
\thanks{Submitted August 4, 2021\funding{This  work  was  supported  by  the  U.S.  Department  of  Energy,  Office  of  Science,  Office  of  Advanced  Scientific Computing  Research,  Scientific  Discovery  through  the  Advanced  Computing  (SciDAC)  Program  through  the FASTMath Institute under Contract No.\ DE-AC02-06CH11357.
Stefan Ulbrich received support by the German Research Foundation (DFG) within the Collaborative Research
Center TRR 154 Project-ID 239904186 - TRR 154 "Mathematical Modelling,
Simulation and Optimization using the Example of Gas Networks", project A02,
and by the DFG within the Collaborative Research Center SFB 1194 Project-ID 265191195 - SFB 1194 "Interaction between Transport and Wetting Processes", project
B04.}}}

\begin{document}
\maketitle

\begin{abstract}
In PDE-constrained optimization, proper orthogonal decomposition (POD)
provides a surrogate model of a (potentially expensive) PDE discretization,
on which optimization iterations are executed.
Because POD models usually provide good approximation
quality only locally, they have to be updated during optimization. 
Updating the POD model is usually expensive, however, and therefore often impossible
in a model-predictive control (MPC) context. Thus, reduced models of
mediocre quality might be accepted.  We take the view of a simplified Newton
method for solving semilinear evolution equations to derive an algorithm that
can serve as an \emph{offline phase} to produce a POD model. Approaches that
build the POD model with  \emph{impulse response snapshots} can
be regarded as the first Newton step in this context.

In particular, POD models that are based on impulse response snapshots
are extended by adding a second simplified Newton
step.  This procedure improves the approximation quality of the POD model
significantly by introducing a moderate amount of extra computational costs
during optimization or the MPC loop.  We illustrate our findings with an
example satisfying our assumptions.
\end{abstract}



\section{Introduction}\label{intro}

\ac{POD} is a well-known method to derive low-dimensional reduced-order models of
dynamical systems.
In the field of optimization of \acp{PDE}, \ac{POD} is employed
as a snapshot-based model order reduction technique to replace  expensive 
\ac{FEM} solves of a discretized PDE by computationally cheap
surrogates in the optimization iterations; see, for example,
\cite{afanasiev2001adaptive,kunisch1999control,kunisch2001galerkin,volkwein2001optimal,kunisch2002galerkin,arian2000trust,bergmann2008optimal,sachs2010pod}.
Because the control inputs change during the optimization, the quality of the
reduced-order model usually deteriorates, and an update or recomputation
may become necessary \cite{afanasiev2001adaptive,arian2000trust}. We
propose a \ac{POD} model that provides increased accuracy for
varying controls compared with common snapshot-based approaches.

We summarize the rationale of \ac{POD} and refer to
\cite{kunisch2002galerkin,benner2020model} for details.
\C{Let $H$ be a Hilbert space and $C([0,T], H)$ be the solution space
of an evolution equation. Note that $H$ can also be replaced by a
suitable finite-dimensional finite element subspace.}
For a given solution $y$ of the evolution equation, let
$\spn\{ y(t)\,|\,0 \le t \le T\} \subset H$ be the subspace of interest,
where $H$ is a Hilbert space.
\C{, which can often be approximated  well by a low-dimensional subspace.}%
Now assume that the span of a set of vectors $\{v_1,\ldots,v_n\} \subset H$
approximates $\spn\{y(t)\,|\, 0 \le t \le T\}$  well  
for trajectories $y$ of our interest. Then we can compute a (reduced) basis of length $k \le n$, which minimizes the squared reconstruction error of the vectors $v_1,\ldots,v_n$, by solving 
\begin{align*}
	\min_{\psi^1,\ldots,\psi^k} \frac{1}{2} \sum_{\ell=1}^n \left\|v_\ell - \sum_{i=1}^k (\psi^i, v_\ell)_H \psi^i \right\|_H^2 
	\text{ s.t. } (\psi^i, \psi^j)_H = \delta_{ij} \text{ for } 1 \le i \le j \le k.
\end{align*}
Such vectors $v_1,\ldots,v_n$ are called \emph{snapshots}. We can solve the
least-squares problem with the help of a \ac{SVD} of the correlation matrix
\[ K = \left( ( v_i, v_j)_H \right)_{1 \le i,j \le n} \]
and a suitable transformation into $H$ yielding $n$ basis vectors (see \cite[Sec.\,3]{kunisch2002galerkin}).
The resulting reduced basis vectors $\psi^1,\ldots,\psi^n$ in $H$ and the 
associated singular values $\lambda_1,\ldots,\lambda_n$ that are given by the 
aforementioned \ac{SVD} satisfy
\[ \sum_{\ell=1}^n \left\|v_\ell - \sum_{i=1}^k (\psi^i, v_\ell)_H \psi^i \right\|_H^2 = \sum_{j=k+1}^n \lambda_j,\]
(see \cite[Sec.\,3]{kunisch2002galerkin}), which allows for a trade-off between reconstruction accuracy and the number of basis vectors. The
computations are faster for fewer basis vectors $k$.
If linear system solves are the bottleneck in the numerical 
computations, the steps in the optimization procedure that
employ the \ac{POD} model have a 
complexity of $\mathcal{O}(k^3)$.

The selection of snapshots is crucial when building the \ac{POD} model.
The state iterates move through the state 
space during optimization, reducing the approximation quality of \ac{POD} 
models that were computed for snapshots in different regions.
Therefore, the locations of good snapshots may be unknown
when starting an optimization procedure, and different strategies
have been developed to handle this situation.
Hinze and Volkwein \cite{hinze2005proper} optimize the \ac{POD} model 
until convergence, compute additional snapshots, and compute a new model
from the increased snapshot set.
Sachs \etal \cite{arian2000trust,bergmann2008optimal,sachs2010pod} integrate
the update of a \ac{POD} model in a trust-region globalization strategy.
Schmidt \etal \cite{schmidt2013derivative} optimize the \ac{POD} model until
convergence and compute a new model from information at the final
iterate. Bott \cite{bott2015adaptive} uses error estimators in a
multilevel sequential quadratic programming (SQP) method to trigger model updates.
Gubisch and Volkwein \cite{gubisch2017proper} increase the number of 
basis vectors during the optimization iterations.

These adaptive strategies require expensive \emph{offline phases}
that update the model and are succeeded by cheap \emph{online phases} until the
next \emph{offline phase}. However, this approach may be difficult in the
context of \ac{MPC}, where there might not be enough time or compute resources
available for multiple offline phases. Ghiglieri and Ulbrich 
\cite{ghiglieri2014optimal} present an \ac{MPC} problem, for
which they combine uncontrolled and impulse response
snapshots---which can both be computed ahead---and keep the \ac{POD} model fixed during the whole \ac{MPC} loop. 
Their article provides a useful insight: the impulse response snapshots
are a fundamental solution of the linearized \ac{PDE}, thereby incorporating 
properties of convolution representations into the snapshot
ensemble. This is the starting point for our investigations in this work.

Convolution representations using impulse responses or \emph{Green's functions} 
are a common tool for analyzing dynamical systems. Bai and Skoogh 
\cite{bai2006projection} consider the Volterra series
representation of bilinear dynamical systems. They construct reduced models that
match a desired number of moments of the transfer functions of the kernels
of the Volterra series. 
Gu \cite{gu2009qlmor} states that Volterra series-based approaches may suffer from bad approximation quality outside a small region around the expansion
point. To alleviate this problem, he proposes to reformulate the polynomial
nonlinear system into so-called quadratic-linear differential algebraic
equations with larger system size.
Then, reduced models are constructed to match a desired number of moments
of the transfer functions of the reformulated system.
Flagg \etal \cite{flagg2015multipoint}
and Benner \etal \cite{benner2012interpolation,benner2018h2} derive optimality conditions of the corresponding 
approximation problems for bilinear and quadratic
bilinear systems. For example, in \cite{benner2018h2},
a truncated $\mathcal{H}_2$-norm, which includes the first three summands of the Volterra series, of a quadratic bilinear system is minimized. Importantly, the optimality
conditions do not depend on any input data of the system and can be satisfied approximately by the system matrices produced by an efficient iterative algorithm.

We propose a novel approach to improve the approximation quality and increase
the region of good approximation quality. Assume we want to
solve $E(y) = 0$, that is, compute the state for a fixed control. Then we can
compute an approximation $y^{(1)} = \bar{y} + d^{(1)}$ of $y$, where $d^{(1)}$
is the solution of one step of Newton's method,
\begin{align*}
	E_y(\bar{y})d^{(1)} &= -E(\bar{y}),
\end{align*}
and $E_y$ denotes the derivative of $E$ with respect to $y$.
Impulse response snapshots yield a high approximation quality
of the linear subspace, in which $d^{(1)}$ lives. However, the quality may
be poor outside a small neighborhood of the Taylor expansion point $\bar{y}$.
Now, we carry out a simplified second step of the Newton method,
\begin{align*}
	E_y(\bar{y})d^{(2)} &= -E(\bar{y} + d^{(1)}).
\end{align*}
The step is \emph{simplified} because we reuse the linearization
and  update only the right-hand side. We approximate the subspace containing 
$d^{(2)}$ also by means of additional impulse response snapshots.
Using \emph{simplified Newton steps} still yields local convergence and is used
in SQP methods as \emph{second-order corrections}
\cite{fletcher2002nonlinear,wachter2005line}.

\subsection{Contribution}

We formalize the described methodology for a class of semilinear
evolution equations with linear control inputs. We characterize the orbits
of $d^{(1)}$ and $d^{(2)}$ and prove bounds on corresponding \ac{POD} 
approximation errors.
We show how this allows us to compute the enriched \ac{POD}
bases by solving suitable impulse response problems as well as
how the latter occur in time discretizations.
We provide computational results that demonstrate the improved approximation 
properties for a semilinear evolution \ac{PDE} and a tracking-type \ac{OCP} 
constrained by it that fit into our framework of assumptions. 
Furthermore, we outline how \acp{MIOCP} may benefit from the use of
the proposed method.

\subsection{Structure of the paper}

In \cref{sec:simplified_newton} we recall the simplified Newton method
and its local convergence properties.
In \cref{sec:evolution_equations} we analyze $d^{(1)}$ and $d^{(2)}$,
their orbits and the \ac{POD} approximation errors.
\Cref{sec:galerkin} transfers the resulting approximation errors to
a Galerkin ansatz for the PDE on the \ac{POD} model.
\Cref{sec:augmented_pod_basis_computation} states
an algorithm that executes the two investigated Newton steps to compute
a combined enriched \ac{POD} basis.
\Cref{sec:computational_results} presents computational results.
In \cref{sec:miocp} we outline how the proposed method can
be used to solve relaxations of \acp{MIOCP}.
We give concluding remarks in \cref{sec:conclusion}.

\section{Simplified Newton method}\label{sec:simplified_newton}

We begin by stating the simplified Newton method, which is defined for
initial vectors $y^{(0)} \in Y$ and operators $F : Y \to Z$ satisfying
\cref{ass:simplified_newton_iteration}.

\begin{assumption}\label{ass:simplified_newton_iteration}
	Let $Y$, $Z$ be Banach spaces, let $F : Y \to Z$ be continuously Fr\'{e}chet differentiable
        on an open convex neighborhood $D$ of $y^{(0)} \in Y$, 
	and let $F'(y^{(0)})\in \mathcal{L}(Y,Z)$ be invertible.
\end{assumption}

\begin{algorithm}
\caption{Simplified Newton method}\label{alg:simplified_newton_iteration}  
\begin{algorithmic} 
	\REQUIRE{$F$ satisfying \cref{ass:simplified_newton_iteration},
	$y^{(0)} \in Y$}
	\FOR{$k = 1,\ldots$}
		\STATE{$d^{(k)} \gets \solve(F'(y^{(0)})d^{(k)} = -F(y^{(k-1)}))$}
		\STATE{$y^{(k)} \gets y^{(k-1)} + d^{(k)}$}		
	\ENDFOR
\end{algorithmic}
\end{algorithm}

The following local convergence result is, for example, shown in
\cite[Sec.\,4.2]{jay2000inexact}.

\begin{proposition}\label{prp:simplified_newton_local_convergence}
	Let $y^* \in Y$ be such that $F(y^*) = 0$, let $F$ be continuously differentiable in an open neighborhood of $y^*$, and
	and let $F'(y^*)\in \mathcal{L}(Y,Z)$ be invertible. Then there exists $\delta > 0$ such that for all
$y^{(0)} \in B_\delta(y^*)$, the iterates 
	$(y^{(k)})_{k\in\mathbb{N}}$ produced by \cref{alg:simplified_newton_iteration} satisfy
	$\left\|y^{(k+1)} - y^*\right\|_Y \le c \left\|y^{(k)} - y^*\right\|_Y$
	for some $0 < c < 1$.
\end{proposition}

Now, we state the approximation of the zero of the state equation achieved by the simplified Newton iteration.

\begin{proposition}
	Let \cref{ass:simplified_newton_iteration}
	and the Lipschitz condition
	\[ \|F'(y^{(0)})^{-1}(F'(y) - F'(y^{(0)}))\|_{\mathcal{L}(Y,Y)} \le \omega_0 \|y - y^{(0)}\|_Y \quad\forall\, y\in D
	\]
	hold for some $\omega_0 > 0$ with
	$h_0 \coloneqq \omega_0 \|d^{(1)}\|_Y < 0.5$.
	Let $\overline{B_r(y^{(0)})} \subset D$ for
	$r = (1 - \sqrt{1 - 2 h_0})/\omega_0$.	
	Then there exists $0 < c < 1$ such that $\|d^{(k+1)}\|_Y \le c \|d^{(k)}\|_Y$, 	
	and for all iterations $k$ it holds that
	\begin{gather*}
		\|F(y^{(k)})\|_Z = \mathcal{O}(\|d^{(k+1)}\|_Y),
		\text{ and }
		\|F(y^{(k)})\|_Z \le \|F'(y^{(0)})\|_{\mathcal{L}(Y,Z)} c^k \|d^{(1)}\|_Y.
	\end{gather*}
\end{proposition}
\begin{proof}
	The iteration reads $-F(y^{(k)}) = F'(y^{(0)})d^{(k+1)}$.
	This implies  $\|F(y^{(k)})\|_Z = \mathcal{O}(\|d^{(k+1)}\|_Y)$.
	The constant $c$ and the estimate follow from \cite[Thm\,2.5]{deuflhard2011newton}. 
\end{proof}

\section{Application to evolution equations}\label{sec:evolution_equations}

We analyze iterations $k=1$ and $k=2$ of \cref{alg:simplified_newton_iteration} 
for a class of evolution equations. The state equation is $E(y, u) = 0$,
where $E : Y \times U \to Z$ satisfies the following assumption.
\begin{assumption}\label{ass:simplified_newton_iterationE}
	Let $U,Y,Z$ be Banach spaces, and let $E : Y\times U \to Z$ be continuously
	Fr\'{e}chet differentiable and linear with respect to $u$. Moreover,
	for all $(\bar y,\bar u) \in Y\times U$ let
	$E_y(\bar y,\bar u)\in \mathcal{L}(Y,Z)$ be invertible.
\end{assumption}
Now fix some $(\bar y,\bar u) \in Y\times U$. We will later choose
$\bar y$ constant in time; see \cref{ass:setting}.
Moreover, often it makes sense to choose
$(\bar y,\bar u)$ as a steady-state solution, that is, $E(\bar y,\bar u)=0$,
but this is not required. Let some control $u\in U$ be given.
In order to compute a solution of $E(y,u)=0$, the first two steps of
\cref{alg:simplified_newton_iteration} are
\begin{align}
	E_y(\bar{y}, \bar{u})d^{(1)} &= -E(\bar{y}, u), & & y^{(1)} \coloneqq \bar{y} + d^{(1)},\text{ and} \label{eq:sl_nwt1} \\
	E_y(\bar{y}, \bar{u})d^{(2)} &= -E(y^{(1)}, u), & & y^{(2)} \coloneqq y^{(1)} + d^{(2)}. \label{eq:sl_nwt2}
\end{align}

In the following, the operator equation $E(y,u)=0$ represents a semilinear parabolic problem of the form
\begin{gather}\label{eq:nlse}
   \partial_t y(t)-A y(t)+ N(y(t))=F u(t),\quad y(0)=y_0, 
\end{gather}
where $A$ is an elliptic spatial operator, $N$ is a nonlinear term
of lower
order, and $F$ is a control operator. An appropriate setting to ensure
\cref{ass:simplified_newton_iterationE} will be given below for
particular examples.
\C{
The space $W$ is the suitable Bochner space for the considered evolution equation. We restrict ourselves to
semilinear evolution equations with differential operators generating strongly continuous semigroups. 
Therefore, we consider the state space of mild solutions $C([0,T], Y)$ in the remainder, 
where $Y$ denotes the spatial state space. Since an inner product is mandatory for the \ac{POD} methodology to 
work, we assume $Y$ to be a Hilbert space. 

Because the control enters linearly, we abbreviate $E_y(\bar{y}) \coloneqq E_y(\bar{y}, \bar{u})$ and $E_u \coloneqq E_u(\bar{y}, \bar{u})$, which removes
nonexisting dependencies from the notation.
We note that it is usually useful for optimization purposes
to consider a setting of evolution equations that satisfy maximal parabolic regularity and to consider 
solutions in spaces 
\[ W^{1,p} = \{ y \in L^p((0,T), V), \partial_t y \in L^{p'}((0,T), V^*) \} \hookrightarrow^c C([0,T], Y) \]
for $p > 1$, where $p'$ solves $\frac{1}{p} + \frac{1}{p'} = 1$ and $V$ is a Hilbert space that satisfies
$V \hookrightarrow^c Y \hookrightarrow^c V^*$.}

\subsection{Guiding example}\label{sec:guiding_example}

The following semilinear \ac{IBVP} serves as our guiding example throughout
the remainder of the article.
\begin{gather}\label{eq:ex_sl}
\left\{
\begin{aligned}
	\partial_t y(t) - a\Delta y(t) + b y(t)^3 - Fu(t) &= 0\quad &&\mbox{on $(0,T)\times
\Omega$},\\
        y    &= 0 \quad &&\mbox{on $(0,T)\times\partial\Omega$},\\
	y(0) &= y_0\quad &&\mbox{on $\Omega$}.
\end{aligned}\right.
\end{gather}
Here, $a,b>0$, $\Delta$ denotes the Dirichlet Laplacian, and $\Omega\subset
\R^d$, $d\in\{1,2,3\}$,
is an open domain that is convex or of class $C^2$. We set
$V \coloneqq H_0^1(\Omega)\cap H^2(\Omega)$,
$W \coloneqq H^1_0(\Omega)$,
$H \coloneqq L^2(\Omega)$,
and $\mathcal{H} \coloneqq
\left\{ y \in L^2(0,T;V)\,\vert\,\partial_t y \in L^2(0,T;H)\right\}$.
We work with the following data and spaces:
\begin{gather}\label{eq:ex_setting}
\begin{aligned}
&F \in H,\quad y_0 \in W,\quad
U = L^2(0,T),\quad Z = L^2(0,T;H)\times W,\\
&Y = \{y\in L^\infty(0,T;W) \cap \mathcal{H}\,|\, y(0)\in W\},\\
&\|y\|_Y=\|y\|_{L^\infty(0,T;W)}+\|y\|_{\mathcal{H}}
+\|y(0)\|_{W}.
\end{aligned}
\end{gather}
Note that $Y$ is well defined because of the continuous embedding 
${\mathcal H} \hookrightarrow C([0,T];W)$; see \cref{app:embedding}.
We have the following existence and uniqueness result.
\begin{proposition}\label{prp:ex_setting}
Consider the setting \eqref{eq:ex_setting}.
Then for any $u\in L^2(0, T)$ there exists a
unique solution $y\in Y
\hookrightarrow C([0,T]; W)$ of \eqref{eq:ex_sl}.
\end{proposition}
\begin{proof}
We apply \cite[Prop.\ 5.1]{barbu2010nonlinear} with
$\beta=\partial g$, where $g: \mathbb{R}\to\mathbb{R}$,
$g(s)=b s^4/4$, and $f = F u\in L^2(0,T; H)$.
Then $\beta : \mathbb{R}\to\mathbb{R}$, $\beta(s)=b s^3$
is continuous and monotonically nondecreasing, and thus the induced graph
is maximally monotone, where the domain satisfies $\overline{D(\beta)} =
\mathbb{R}$ and thus satisfies the assumptions of \cite[Prop.\ 5.1]{barbu2010nonlinear}.
Since the embedding $W\hookrightarrow L^6(\Omega)$ is
continuous for $d=1,2,3$, we have $g(y_0)\in L^1(\Omega)$.
Now \cite[Prop.\ 5.1]{barbu2010nonlinear}
and the continuous embedding
${\mathcal H} \hookrightarrow C([0,T]; W)$
yield the assertion.\qed
\end{proof}

A formal calculation shows that the following problems have
to be solved for the two simplified Newton steps given in \eqref{eq:sl_nwt1},
and \eqref{eq:sl_nwt2}:
\begin{align}
	\left\{ \begin{array}{rl}\label{eq:ex_sl_step1}
		(\partial_t - a\Delta + 3b\bar{y}^2)d^{(1)} &= 
-(\partial_t\bar y - a\Delta\bar y + b\bar{y}^3-Fu), \\ 
		d^{(1)}(0) &= y_0-\bar y(0),\quad d^{(1)}|_{(0,T)\times \partial\Omega}=0
    \end{array}\right.
    \\
    \left\{\begin{array}{rl}\label{eq:ex_sl_step2}
		(\partial_t - a\Delta + 3b\bar{y}^2)d^{(2)} &= -(b(y^{(1)})^3 - b\bar{y}^3 - 3b\bar{y}^2d^{(1)}), \\ 
		d^{(2)}(0) &= 0,\quad d^{(2)}|_{(0,T)\times \partial\Omega}=0.
	\end{array}\right.
\end{align}

To make this rigorous, we introduce the operator $E$:
\begin{equation}\label{eq:Edef}
 E: Y\times U \to Z,\quad
 E(y,u) \coloneqq \binom{\partial_t y - a\Delta y + b y^3 - Fu}{y(0)-y_0},
\end{equation}
and show that it satisfies \cref{ass:simplified_newton_iterationE}.

\begin{proposition}\label{prp:sl_derivative}
Let $E$ be given by \eqref{eq:Edef}. Then for any $\bar u\in U$
the equation $E(\bar y,\bar u)=0$ has a unique
solution $\bar y\in Y$. Moreover, $E$ is continuously
Fr\'echet differentiable. For $(y,u) \in Y \times U$
and any $(v,w)\in Y\times U$ it holds that
\begin{gather*}
 E_y(y,u) v=\binom{\partial_t v - a\Delta v + 3 b y^2 v}{v(0)},
 \quad\text{and}\quad
 E_u(y,u) w=\binom{-Fw}{0}.
\end{gather*}
Moreover, $E_y(y,u)\in {\mathcal L}(Y,Z)$ has a bounded inverse.
\end{proposition}
\begin{proof}
By the definition of $Y,U,Z$ all linear parts of
\eqref{eq:Edef} are in ${\mathcal L}(Y,Z)$. Moreover, the mapping
$N: y\in Y \mapsto y^3\in L^2(0,T;H)$
is continuously Fr\'echet differentiable. In fact, the trilinear form
$B: (y_1,y_2,y_3)\in Y\times Y\times Y \mapsto y_1 y_2 y_3 \in L^2(0,T;H)$
is  bounded because
\begin{align*}
 \|y_1(t) y_2(t) y_3(t)\|_{H} &\le 
\|y_1(t)\|_{L^6(\Omega)} \|y_2(t)\|_{L^6(\Omega)} \|y_3(t)\|_{L^6(\Omega)}\\
&\le C^3_1 \|y_1(t)\|_{W} \|y_2(t)\|_{W} \|y_3(t)\|_{W}
\end{align*}
for some $C_1 > 0$ and thus
\[
 \|y_1 y_2 y_3\|_{L^2(0,T;H)}\le \sqrt{T} \|y_1 y_2 y_3\|_{L^\infty(0,T;H)}
 \le \sqrt{T} C^3 \|y_1\|_{Y} \|y_2\|_{Y} \|y_3\|_{Y}.
\]
Hence, $B$ is infinitely many times continuously Fr\'echet differentiable,
and by the chain rule also $N(y)=B(y,y,y)$ with derivative
$v\in Y \mapsto 3 B(y,y,v)=3 y^2 v\in L^2(0,T;H)$.

Furthermore, since $3 b y^2\in L^\infty(0,T;L^3(\Omega))$, we have for
$v_1$, $v_2\in W$
\[ \int_\Omega v_1 3 b y(t)^2 v_2\,dx
\le 3b C^2_2 \|y\|_Y^2 \|v_1\|_{L^3(\Omega)} \|v_2\|_{L^3(\Omega)}
\le 3b C^2_2 C_3^2 \|y\|_Y^2 \|v_1\|_{W} \|v_2\|_{W}
\]
for some $C_2$, $C_3 > 0$. Hence,
\[
  a(t;v_1,v_2)=(\nabla v_1,\nabla v_2)_{H^d}+\int_\Omega v_1 3 b y(t)^2
v_2\,dx
\]
defines uniformly in $t$ a bounded and coercive bilinear form on
$W\times W$. Standard parabolic theory yields
a unique solution of $E_y(y,u)v=z$ for all $z\in Z$ with
$v\in \mathcal{W} \coloneqq \{w\in L^2(0,T;W)\,|\,\partial_t w\in
L^2(0,T;H^{-1}(\Omega))\}$ with $\|v\|_\mathcal{W} \le C\|z\|_Z$
for a constant $C$ independent of $z$.
Then $3 b y^2 v\in L^2(0,T; H)$ and thus $v\in Y$ by \cref{prp:ex_setting}. Hence, $E_y(y,u)^{-1}\in {\mathcal L}(Z,\mathcal{W} )$ and
$E_y(y,u)^{-1}: Z \to Y$. Now $E_y(y,u)^{-1}\in {\mathcal L}(Z,Y)$
follows from the closed graph theorem. Alternatively, one can apply standard
parabolic regularity theory; see, for example,  \cite[7.1, Thm. 5]{evans1998partial}.
\end{proof}
Hence, for the setting \eqref{eq:ex_setting}
we have justified that the formally derived
simplified Newton steps in \eqref{eq:sl_nwt1} and \eqref{eq:sl_nwt2}
are well defined and have the desired regularities.

Since $by^3 \in L^2(0,T;H)$ for $y\in Y$, the uniqueness of the mild
solution of $E(y,u)=0$ and a bootstrapping argument imply that $y$
can be represented by the variation of constants formula
\begin{align}\label{eq:ex_sl_se}
	y(t) = S(t)y_0 + \int_0^t S(t - s)(Fu(s) - by^3(s))\dd s,
\end{align}
where $(S(t))_{t\ge 0}$ denotes the strongly continuous semigroup generated by the
Dirichlet Laplacian (scaled by $a > 0$) on $H$. Similarly, the solution
$v\in Y$ of $E_y(y,u)v = (z,0)$ for $(z,0)\in Z$ can be represented by
\begin{align}\label{eq:ex_sl_le}
v(t) = \int_0^t S(t - s)(z(s) - 3 by^2(s)v(s))\dd s.
\end{align}

\subsection{General framework}\label{sec:general_framework}

This section provides our standing assumptions on the considered
\acp{IBVP} \eqref{eq:nlse}.
The guiding example presented above meets the assumptions.
We associate with \eqref{eq:nlse} an operator $E$,
\begin{equation}\label{eq:Edefgen}
 E: Y\times U \to Z,\quad
 E(y,u)=\binom{\partial_t y - Ay + N(y) - Fu}{y(0)-y_0},
\end{equation}
where boundary conditions are included in the definition of $Y$.
We will work under the following assumption.

\begin{assumption}\label{ass:setting}
        Let $V\hookrightarrow W \hookrightarrow H$ be Hilbert spaces
        with dense imbeddings. Assume that with $y_0\in W$, $F\in H$,
        $U=L^2(0,T)$ and appropriate spaces
        $Y\hookrightarrow L^2(0,T;V)\cap C([0,T],H)$, 
        $Z \hookrightarrow L^2(0,T;H) \times W$ the operator
        $E$ defined in \eqref{eq:Edefgen} satisfies
        \cref{ass:simplified_newton_iterationE}.
        Moreover, for any time-independent state $\bar y\in Y$ the operator
        $B:=A-N_y(\bar y): D(B) \to H$ generates a strongly continuous
        semigroup $(T(t))_{t \ge 0}$.
\end{assumption}

\begin{remark}\mbox{}
\begin{itemize}
\item If $A: D(A)\to H$ generates a strongly continuous semigroup
      and $C:=-N_y(\bar y) \in \mathcal{L}(H,H)$, then $B=A + C$ generates a
      strongly continuous semigroup on $H$; see, for example, 
      Corollary 3.5.6 in \cite{arendt2011vector}.
      Thus $(T(t))_{t \ge 0}$ is well defined.\\
      However, any setting where $A + C$ generates
	  a strongly continuous semigroup is allowed.
	  This is, for example, the case for the perturbation
	  $C = \bar{y} \cdot \nabla$ (Oseen semigroup)
	  if $A$ is the Stokes operator; see
	  \cite{miyakawa1982nonstationary}.
\item By \cref{ass:setting}, for any time-independent state
      $\bar y\in Y$ and all $\bar u\in U$, $(z,v_0)\in Z$ the linearized
      equation $E_y(\bar y,\bar u)v=(z,v_0)$
      has a unique solution $v\in Y$. Semigroup theory allows one to represent
      $v$ as the unique mild solution
      \begin{equation}\label{eq:mild}
         v(t)=T(t) v_0+\int_0^t T(t-s) z(s)\,ds.
      \end{equation}
\item It makes sense to consider linearizations at stationary states
      $\bar y$. For example, since $a > 0$ and $b > 0$ in our guiding example
      we may expect a damping behavior toward a stationary state. Moreover,
      in many applications a stabilization around a stationary state by
      optimal control is relevant.
\end{itemize}
\end{remark}

\subsection{A convolution formula for the first Newton step} 

We investigate the first simplified Newton step
and derive a convolution formula for $d^{(1)}(t)$.
We fix a linearization point $\bar y \in Y$ that is constant in time,
that is, $\bar y(t)=\bar y_0$ for some $\bar y_0\in V$.

The first Newton step \eqref{eq:sl_nwt1} for \eqref{eq:Edefgen}
 with $C=-N_y(\bar y)$ is
\begin{align} \label{eq:newt1}
 \binom{(\partial_t - A - C) d^{(1)}}{d^{(1)}(0)}=
\binom{A\bar y-N(\bar y)+F u}{y_0 - \bar{y}_0}.
\end{align}
We will show that the solution $d^{(1)}\in Y$ of \eqref{eq:newt1}
can be computed from the solutions of the following
problems:
\begin{align} \label{eq:y0_offset}
 \binom{(\partial_t - A - C) v}{v(0)}&=\binom{A\bar y-N(\bar y)}{y_0 -
\bar{y}_0},\text{ and}\\
\label{eq:ir0}
 \binom{(\partial_t - A - C) w}{w(0)}&=\binom{0}{F}.
\end{align}
Because of the structure of \eqref{eq:ir0}, we call $w \in C([0,T];H)$
the impulse response for the right-hand side (impulse) $F$.
Now the following holds.

\begin{lemma}\label{lem:yirscf}
  Let \cref{ass:setting} hold. 
  Let $v\in Y$ solve \eqref{eq:y0_offset}, and let
  $w\in C([0,T];H)$ solve \eqref{eq:ir0}. 
  Then the solution $d^{(1)}\in Y$ of \eqref{eq:newt1} is given by
  \begin{align*}
    d^{(1)}(t) &= v(t) + \int_0^t w(t - s)u(s) \dd s.
  \end{align*}
  If $F\in W$, then we have $w\in Y$.
\end{lemma}
\begin{proof}
By \cref{ass:setting} the problem \eqref{eq:newt1} has a unique solution
$d^{(1)}\in Y$ that can also be represented as a mild solution (see
\eqref{eq:mild}):
      \[
         d^{(1)}(t)=T(t) (y_0-\bar y_0)+\int_0^t T(t-s) (A\bar y-N(\bar y)+Fu(s))\dd s.
      \]
Now set $v(t)=T(t) (y_0-\bar y_0)+\int_0^t T(t-s) (A\bar y-N(\bar y))\dd s$
and $w(t)=T(t) F$. Then $v$ is the mild solution of \eqref{eq:y0_offset},
$w$ is the mild solution of \eqref{eq:ir0}, and the claimed representation
of $d^{(1)}(t)$ follows. Moreover, if $F\in W$, then \cref{ass:setting}
implies that $w$ solving \eqref{eq:ir0} is in $Y$.
\end{proof}

\subsection{A discrete convolution formula for the first Newton step}\label{sec:first_newton_step_time_discrete}

We consider now a $\theta$-scheme for time discretization that comprises the
implicit Euler scheme ($\theta=1$) and the Crank--Nicolson scheme
($\theta=1/2$).

Let $0 = t_0 < \ldots < T_K = T$,
$\Delta t = t_{k+1} - t_k$ be a uniform time grid, and $(u_k)_{k\in
\{0,\ldots,K - 1\}}$ be an interval-wise constant
discretization of the control $u$. We approximate \eqref{eq:nlse} by
\[
  \frac{y_{k+1}-y_k}{\Delta t}-A (\theta y_{k+1}+(1-\theta) y_k)
+\theta N(y_{k+1})+(1-\theta) N(y_k)+F u_k,~~
0\le k <K,
\]
where $\theta\in [\frac{1}{2},1]$.
 Then the discrete analogue of \eqref{eq:newt1} is
\begin{align}
\notag
  \frac{d^{(1)}_{k+1}-d^{(1)}_k}{\Delta t}-(A+C) (\theta d^{(1)}_{k+1}+(1-\theta) d^{(1)}_k)&=
  A\bar y-N(\bar y)+F u_k, ~
0\le k<K,\\
\label{eq:d1k}
d^{(1)}_0&=y_0 - \bar{y}_0,
\end{align}
and the one of \eqref{eq:y0_offset} is
\begin{align}
\label{eq:vk}
\begin{split}
  \frac{v_{k+1}-v_k}{\Delta t}-(A+C) (\theta v_{k+1}+(1-\theta) v_k)&=
  A\bar y-N(\bar y), ~~
0\le k <K,\\ v_0&=y_0 - \bar{y}_0.
\end{split}
\end{align}
Now consider the following discretization of \eqref{eq:ir0}:
\begin{align}
\label{eq:wk}
\begin{split}
  \frac{w_{k+1}-w_k}{\Delta t}-(A+C) (\theta w_{k+1}+(1-\theta) w_k)&=0,\quad
0\le k <K,\\
(I+\Delta t (1-\theta) (A+C)) w_0 &=F.
\end{split}
\end{align}

Then we obtain the following discrete convolution formula for the $\theta$-scheme.

\begin{proposition}\label{prp:conv_theta}
Let $(v_k)$ and $(w_k)$ solve \eqref{eq:vk} and \eqref{eq:wk}. Then the
first Newton step $(d^{(1)}_k)$ for the $\theta$-scheme
\eqref{eq:d1k} can be represented by
\begin{align*}
	d^{(1)}_{k} = v_k +
		\Delta t \sum_{j=0}^{k-1} w_{k - j} u_{j},
\end{align*}
where we use the convention that the sum vanishes for $k=0$.
\end{proposition}
\begin{proof}
Set $e_k \coloneqq \Delta t \sum_{j=0}^{k-1} w_{k - j} u_{j}$.
Then $e_0=0$, and by using \eqref{eq:wk} we have
\begin{align*}
\frac{e_{k+1}-e_k}{\Delta t}&=\sum_{j=0}^{k} (w_{k+1 - j}-w_{k - j}) u_{j}
+w_0 u_k\\
&=(A+C) \Delta t \sum_{j=0}^{k}  (\theta w_{k+1 - j}+(1-\theta)w_{k - j}) u_{j}
+w_0 u_k\\
&=(A+C) (\theta e_{k+1}+(1-\theta) e_k)+(I+(A+C) \Delta t (1-\theta)) w_0 u_k\\
&=(A+C) (\theta e_{k+1}+(1-\theta) e_k)+F u_k.
\end{align*}
Now by superpositon $v_k+e_k$ satisfies \eqref{eq:d1k} as asserted.
\end{proof}

\subsection{Subspace characterization and approximation of the first Newton step}\label{sec:approximation_newt1}

We use the convolution formula to characterize the orbit of the first Newton step
for arbitrary controls $u\in U$. This will be exploited to obtain the
reduced basis for the \ac{POD} model proposed in this article.
To state the result precisely, we need further notation.
For a function $f \in L^p(0,T; X)$ for some Banach space $X$ we
write
\[ f([0,T]) \coloneqq \bigcap_{N \subset [0,T], \lambda(N) = 0}
\overline{\{ f(t) \,|\, t \in [0,T] \setminus N \}}^X, \]
where $\lambda$ denotes the Lebesgue measure and we call $f([0,T])$
the \emph{essential range} of $f$.

\begin{theorem}\label{thm:first_newton_step_spaces}
	Let $\bar{y}$ be a given linearization point
	of $E$, $v$ solve \eqref{eq:y0_offset}, $w$ solve \eqref{eq:ir0},
	$u \in U$ be arbitrary, and $d^{(1)}$ solve \eqref{eq:newt1}. Then
	\begin{align*}
    	d^{(1)}(t) - v(t)
    	\in \overline{\spn w([0,T])}^W.
	\end{align*}
	\C{where $w([0,T])$ denotes the essential range of $w$.}%
	%
	%
	The discrete analogs $(d^{(1)}_k)$, $(v_k)$, and $(w_k)$ for
	the $\theta$-scheme \eqref{eq:d1k}, \eqref{eq:vk}, and \eqref{eq:wk}
	satisfy $d^{(1)}_k - v_k \in \spn \{ w_1,\ldots,w_{k} \}$.
\end{theorem}
\begin{proof}
The trajectories $y_1$, $\bar{y}$ and $z$ are mild solutions and continuous accordingly. Thus, the 
pointwise evaluation makes sense. Employing \cref{lem:yirscf}, we observe that
\[ d^{(1)}(t) - v(t) = \int_0^t w(t - s)u(s) \dd s
= \int_0^t w(s)u(t - s) \dd s
\]
with $u(t - s) \in \mathbb{R}$.
Because $w \in L^2(0,T;W)$ and $u \in L^2(0,T)$, it follows that the integrand
in the convolution formula above is in $L^1(0,T;W)$ for all $t \in [0,T]$.
Thus, a vector-valued version of the mean value theorem for Bochner integrals (see  \cite[Cor.\,II.8]{diestel1977vector})
(after replacing all \emph{for all} statements in
its proof by \emph{for almost all}) yields 
\[ d^{(1)}(t) - v(t) \in t\, \overline{\conv f_t([0,t])}^W, \]
where $f_t \in L^1(0,t; W)$ with $f_t(s) \coloneqq w(s)u(t - s)$ for a.a.\ $s \in [0,t]$ and all $t \in [0,T]$. Because $u$ is $\R$-valued,
it follows that $f_t([0,t]) \subset \spn w([0,T])$,
which closes the argument. The $\spn$ for the discrete trajectories follows
from \cref{prp:conv_theta}.
\end{proof}

To approximate $d^{(1)}$, we consider a POD approximation of $w$
in $L^2(0,T; W)$ of rank $n \in \N$. That is, we seek to bound
the approximation error
\begin{gather}\label{eq:cont_pod_problem}
\begin{aligned}
\min_{\psi^1,\ldots,\psi^n}\ &\frac{1}{2}\int_0^T
   \left\|w(t) - \sum_{i=1}^n (\psi^i, w(t))_W \psi^i\right\|^2_W\,\dd t\\
\text{ s.t.\ } &(\psi^i, \psi^j)_W = \delta_{ij} \text{ for all } 1 \le i \le j \le n. 
\end{aligned}
\end{gather}
To this end, we adapt the \ac{POD} approximation from \cite[Sec.\,3]{kunisch2002galerkin}.
Let the operator $\mathcal{Y} : L^2(0,T;\R) \to W$ be defined as
$\mathcal{Y}\varphi \coloneqq \int_0^T \varphi(t)w(t)\dd t$.
Its adjoint $\mathcal{Y}^* : W \to L^2(0,T;\R)$ is
$(\mathcal{Y}^*f)(t) = (f, w(t))_W$
for a.a.\ $t \in (0,T)$. Defining
$\mathcal{R} \coloneqq \mathcal{Y}\mathcal{Y}^*$ yields
%
\[ \mathcal{R}z = \int_0^T (z, w(t))_W w(t)\,\dd t. \]
Then we can characterize the POD approximation by means of
the spectrum of $\mathcal{R}$.
\begin{proposition}\label{prp:cont_pod_w}
Let the assumptions of \cref{thm:first_newton_step_spaces} hold.
Then there exists an orthonormal basis $(\psi^i)_{i\in\N}$ of $W$
and $(\lambda_i)_{i\in\N} \subset [0,\infty)$ such that
$\mathcal{R}\psi^i = \lambda_i\psi^i$ for all $i \in \N$ and
$\lambda_i \to 0$.
Moreover, it follows that
\[ \int_0^T \|w(t)\|_W^2\dd t = \sum_{i=1}^\infty \lambda_i \]
and for all $n \in \N$ it holds that
\[ \int_0^T \left\|w(t) - \sum_{i=1}^n (\psi^i, w(t))_W \psi^i\right\|^2_W \dd t 
   = \sum_{i=n+1}^\infty \lambda_i. \]
\end{proposition}
\begin{proof}
This follows from the analysis in Section 3 of \cite{kunisch2002galerkin},
in particular the Hilbert--Schmidt theorem applied to $\mathcal{R}$,
with the choice $X = W$ if we are able to show that
the mapping $\mathcal{Y}^*$ is compact (see the 2nd paragraph on page 498
in \cite{kunisch2002galerkin}) for $w \in L^2(0,T;W)$.
Let $f \in B$ and $B \subset W$ be bounded, that is, $C \coloneqq \sup_{f \in B}\|f\|_{W} < \infty$. We obtain
\begin{align*}
\sup_{f \in B}\left|(\mathcal{Y}^*f)(t) - (\mathcal{Y}^*f)(t + h)\right|
&= \sup_{f \in B} |(f, w(t) - w(t + h))_W|\\
&\le C \|w(t) - w(t + h)\|_W
\end{align*} 
for a.a.\ $t$, $t + h \in (0,T)$ using the Cauchy--Schwarz inequality.
Because $\int_0^{T -h} \|w(t) - w(t + h)\|_W^2\dd t \to 0$
for $h \to 0$, it holds that
\begin{multline*}
\sup_{f \in B} \int_0^{T -h} \left|(\mathcal{Y}^*f)(t) - (\mathcal{Y}^*f)(t + h)\right|^2\dd t
\le C^2 \int_0^{T -h} \|w(t) - w(t + h)\|_W^2\dd t \to 0
\end{multline*}
for $h \to 0$, which shows equicontinuity of $\mathcal{Y}^*$ with respect
to $L^2(0,T;\R)$. We can hence apply the Riesz--Kolmogorov
compactness theorem \cite{riesz1933sur,simon1986compact}
to deduce that $\mathcal{Y}^*(B)$ is a compact set,
which implies that $\mathcal{Y}^*$ is a compact operator.
\end{proof}

To use this approximation in the remainder, we introduce the
following notation. Let $\psi^1,\ldots,\psi^n$ be an orthonormal
subset of $W$. Then for $f \in L^2(0,T; W)$, we define the
(pointwise a.e.) orthogonal projection
\[ \PPi_\psi f(t) \coloneqq \sum_{i=1}^n (\psi^i, f(t))_W \psi^i. \]

The argument above does not depend on the function $w$,
and the function $v$ can be approximated analogously. Therefore, we
consider a joint reduced basis for $w$ and $v$ in the remainder.
We consider the projection $\PPi_\psi d^{(1)}$ of $d^{(1)}$ on the
reduced basis
\[ \PPi_\psi d^{(1)} (t)=
\sum_{i=1}^n\left(\psi_i,v(t) + \int_0^t w(s)u(t - s)\,\dd s\right)_W\psi_i.
\]
We denote the corresponding approximation of the first Newton step as
\[ y^{(1)}_\psi \coloneqq \bar{y} + \PPi_\psi d^{(1)} \]
and denote the projection error, which can be driven to zero
by \cref{prp:cont_pod_w}, as
\[ e_1 \coloneqq \|d^{(1)} - \PPi_\psi d^{(1)}\|_{L^2(0,T;W)}. \]
We summarize the resulting approximation quality $y^{(1)}_\psi$ below.

\begin{corollary}
Let the assumptions of \cref{thm:first_newton_step_spaces} hold.
Let $y \in Y$ solve \eqref{eq:nlse}. Then
\[ \|y - y^{(1)}_\psi\|_{L^2(0,T;W)}
   \le \|y - y^{(1)}\|_{L^2(0,T;W)} + e_1.
\]
Let $(\psi^i)_{i \in \N}$ and $(\lambda_i)_{i \in \N}$ be as in
\cref{prp:cont_pod_w}. Then
\[ \|y - y^{(1)}_\psi\|_{L^2(0,T;W)}
   \le \|y - y^{(1)}\|_{L^2(0,T;W)} + \sum_{i=n+1}^{\infty}\lambda_i.
\]
\end{corollary}
\begin{proof}
The claim follows from $y^{(1)} = \bar{y} + d^{(1)}$,
\cref{lem:yirscf}, and \cref{prp:cont_pod_w}
applied to $d^{(1)}$ instead of $w$.
\end{proof}

\begin{remark}
This means that the error of approximating $y$ with a projection of the
first Newton step to the reduced space is bounded by the sum of the
error of the Newton step and the POD approximation error of $v$ and the impulse response $w$, which are both independent of the control input.
\end{remark}

\subsection{Subspace characterization of the second simplified Newton step}\label{sec:second_newton_step}

We consider again the fixed linearization point $\bar y \in Y$,
where $\bar y$ is constant in time, that is, $\bar y(t)=\bar y_0$
for some $\bar y_0\in V$.
We recall the second simplified Newton step
\begin{align*}
	E_y(\bar{y},\bar{u})d^{(2)} &= -E(y^{(1)}, u), & & y^{(2)} \coloneqq y^{(1)} + d^{(2)}. 
\end{align*}
Applying this to the nonlinear operator $E$ defined in \eqref{eq:Edefgen}, we obtain
\begin{align} \label{eq:newt2}
 \binom{(\partial_t - A - C) d^{(2)}}{d^{(2)}(0)}=
\binom{N(\bar{y}) - C d^{(1)} - N(y^{(1)})}{0},
\end{align}
which follows after inserting that the first Newton step $d^{(1)}$
solves \eqref{eq:newt1} and the fact $\partial_t \bar{y} = 0$
into the definition of $E$.
\C{
which follows from
\begin{align*}
E(y^{(1)}, u) &= \partial_t \bar{y} - A \bar{y} + \partial_t d^{(1)} - A d^{(1)} + N(y^{(1)}) - Fu \\
		&= \partial_t \bar{y} - A\bar{y} + \partial_t d^{(1)} - A d^{(1)} - C d^{(1)} + C d^{(1)} + N(y^{(1)}) - Fu \\
		&= \partial_t \bar{y} - A\bar{y}
		+ A\bar{y} - N(\bar{y}) + Fu + C d^{(1)} + N(y^{(1)}) - Fu \\		
		&= - (N(\bar{y}) - C d^{(1)} - N(y^{(1)})
\end{align*}
where we have inserted the first Newton step \eqref{eq:newt1} in the
right hand side of the third equation and used $\partial_t \bar{y} = 0$.}

\begin{lemma}\label{lem:second_newton_step_formula}
	Let \cref{ass:setting} hold. Then the solution
	$d^{(2)} \in Y$ of \eqref{eq:newt2} is given by
	\begin{gather}\label{eq:d2_representation}
	d^{(2)}(t) = \int_0^t T(t - s)\left(N(\bar{y}) - C d^{(1)} - N(y^{(1)})\right)\,\dd s.
	\end{gather}
\end{lemma}
\begin{proof}
\Cref{ass:setting} implies that \eqref{eq:newt2} has
a unique solution $d^{(2)} \in Y$ that is a mild solution
and can be represented with the variation of constants formula.
\end{proof}

Next, we assume that an orthonormal subset $\psi^1,\ldots,\psi^M \subset W$
is given, and we aim to characterize the solution $d^{(2)}[\psi]$
of \eqref{eq:newt2} for the case that $d^{(1)}$ and
$y^{(1)}$ have been replaced by the approximations
$y^{(1)}_\psi$ and $\PPi_\psi d^{(1)}$ obtained
in \cref{sec:approximation_newt1}.

We restrict our analysis to the case that $N$ is the superposition operator
of a polynomial with degree $p \in \N$. The set of monomials
$\{1,x,\ldots,x^p\}$ constitutes a basis of the polynomials, which implies
\[ N(y^{(1)}_\psi)(t)
   \in \bigcup_{i=1}^p
   \spn\left\{ \Pi_{j=1}^i b_j\,\big|\,
    b_1,\ldots,b_i \in \{\bar{y},\psi^1,\ldots,\psi^n\}\right\}   
    \eqqcolon \mathcal{C},
\]
where we further require that $\mathcal{C} \subset H$ and deduce
that there are orthonormal vectors $c_1,\ldots,c_m$---e.g.\ obtained
by Gram--Schmidt orthonormalization---such that we may write
$\mathcal{C} = \spn\{ c_1,\ldots,c_m \}$.
Note that $\mathcal{C} \subset H$ is satisfied for our guiding
example because of the continuous embedding
$W \hookrightarrow L^6(\Omega)$. In particular we obtain
\[ N(y^{(1)}_\psi)(t) = \sum_{j=1}^m \left(c^j, N(y^{(1)}_\psi)(t)\right)_H
c^j \]
for all $t \in [0,T]$.

We consider \eqref{eq:newt2} with the approximations $\PPi_\psi d^{(1)}$ and
$y^{(1)}_\psi$ substituted for $d^{(1)}$ and $y^{(1)}$.
Then the representation \eqref{eq:d2_representation} and the subspaces
$\spn\{\psi^1,\ldots,\psi^n\}$ and $\mathcal{C}$ give rise to the 
initial value problems
\begin{align} 
\left\{ \binom{(\partial_t - A - C) \beta^i}{\beta^i(0)}=
\binom{0}{-C \psi^i}\right. \text{ for } i \in \{1,\ldots,n\},\label{eq:newt2_b}\\
\left\{ \binom{(\partial_t - A - C) \gamma^j}{\gamma^j(0)}=
\binom{0}{-c_j}\right. \text{ for } j \in \{1,\ldots,m\}. \label{eq:newt2_c}
\end{align}
Similar to the first Newton step, we can now characterize the subspace
that contains the orbit of the second simplified Newton step
if $d^{(1)}$ has already been reduced by means of a POD approximation.

\begin{theorem}\label{thm:second_newton_step_spaces}
Let $\bar{y}$ be a given linearization point
of $E$.
Let $N$ be the superposition operator of a polynomial of
such that $N \in C(W, H)$.
Let $\{\psi^1,\ldots,\psi^n\}$, $\mathcal{C}$,
\eqref{eq:newt2_b}, and \eqref{eq:newt2_c} be as introduced above.
Let $d^{(2)}[\psi]$ solve \eqref{eq:newt2}
with the approximations $\PPi_\psi d^{(1)}$ and
$y^{(1)}_\psi$ substituted for $d^{(1)}$ and $y^{(1)}$.
Then for all $t \in [0,T]$ we have
\begin{align*}
d^{(2)}[\psi](t) = r(t) + b(t) + c(t)
\end{align*}
with
\begin{align*}
r(t) &= \int_{0}^t T(t - s)N(\bar{y})\,\dd s,\\
b(t) &\in \overline{\spn\left\{ \bigcup_{i=1}^n \beta^i([0,T]) \right\}}^W,
\text{ and}\\
c(t) &\in \overline{\spn\left\{ \bigcup_{j=1}^m \gamma^j([0,T])\right\}}^W,
\end{align*} 
where $\beta^i([0,T])$ and $\gamma^j([0,T])$ denote the essential ranges
of $\beta^i$ and $\gamma^j$.
\end{theorem}
\begin{proof}
With the analysis above we have that \eqref{eq:newt2_b}
and \eqref{eq:newt2_c} admit unique solutions $\beta_i$,
$\gamma_j \in L^2(0,T;W)$ for $i \in\{1,\ldots,n\}$, and
$j\in\{1,\ldots,m\}$. \Cref{lem:second_newton_step_formula}
implies that $d^{(2)}[\psi](t) = r(t) + b(t) + c(t)$ holds with $r$ as claimed:
\[ b(t) = \int_0^{t} \sum_{i=1}^n\beta^i(s)u^i(t - s)\dd s,
   \quad\text{and}\quad
   c(t) = \int_0^{t} \sum_{j=1}^m\gamma^j(s)v^j(t - s)\dd s.
\]
Repeating the argument from the proof of \Cref{thm:first_newton_step_spaces},
we obtain
\[ b(t) \in t\, \overline{\conv\left\{
\bigcup_{i=1}^n f_t^i([0,t])\}
\right\}}^W\quad\text{and}\quad
c(t) \in t\, \overline{\conv\left\{
\bigcup_{j=1}^m g_t^j([0,t])\right\}}^W,
\]
where $f_t^i(s) \coloneqq \beta^i(s)u^i(t - s)$
and $g_t^j(s) \coloneqq \gamma^j(s)v^j(t - s)$ for a.a.\ $s \in [0,t]$
and all $t \in [0,T]$.
\end{proof}

\begin{remark}\label{rem:pod_reduction_of_C}
$\mathcal{C}$ is generated by the sets of $k$-combinations
(for $k = 1,\ldots,p+1$) of (basis) vectors $\bar{y},\psi^1,\ldots,\psi^n$,
which grows excessively with $n$ and $p$. Therefore it may be advisable
to reduce the basis $c^1,\ldots,c^m$ with POD as well.
\end{remark}

\subsection{Subspace approximation of the second simplified Newton step}\label{sec:approximation_newt2}

We consider $\psi^1,\ldots,\psi^n$ and $e_1$ as
in \cref{sec:approximation_newt1}. The error estimates below
depend on the approximation error of the nonlinear operator
$N$ at $y^{(1)}$, which we define as
\[ \ell(e_1,y^{(1)}) \coloneqq
   \|N(y^{(1)}) - N(y^{(1)}_\psi)\|_{L^2(0,T;H)}. \]
We briefly show how an estimate on $\ell(e_1,y^{(1)})$
can be derived for our guiding example.

\begin{example}\label{exa:ell_e1_y1}
We consider the \ac{POD} approximation of $d^{(1)}$ analyzed
in \cref{sec:approximation_newt1} and $N$ defined
as $N(\eta) \coloneqq \eta^3$.
For brevity of the presentation, we assume that
$\bar{y} \in \spn\{\psi^1,\ldots,\psi^n\}$, and we 
define $y \coloneqq y^{(1)}$ and $y_\psi \coloneqq y^{(1)}_\psi$.

For a.a.\ $t\in [0,T]$, we obtain
\begin{align*}
\|y(t)^3 - y_\psi(t)^3\|_{L^2} &\le
\|y(t)^2 + y(t)y_\psi(t) + y_\psi(t)^2\|_{L^3}
\|y(t) - y_\psi(t)\|_{L^6}\\
&\le 3\|y(t)\|_{H^1_0}^2\|y(t) - y_\psi(t)\|_{H^1_0},
\end{align*}
where H\"{o}lder's inequality yields the first inequality.
The second inequality follows from the fact that
$y_\psi(t) = \PPi_\psi y(t)$ and thus
$\|y_\psi(t)\|_{H^1_0} \le \|y(t)\|_{H^1_0}$
and the embedding $H^1_0(\Omega)\hookrightarrow L^6(\Omega)$.
We integrate over both sides and use H\"{o}lder's inequality
to obtain
\begin{align*}
\int_0^T \|y(t)^3 - y_\psi(t)^3\|_{L^2}^2\,\dd t
&\le \int_0^T  3\|y(t)\|_{H^1_0}^4\|y(t) - y_\psi(t)\|_{H^1_0}^2\,\dd t\\
&\le 3\|y\|_{C([0,T];H^1_0)}^4 \int_0^T \|y(t) - y_\psi(t)\|_{H^1_0}^2\,\dd t\\
&\le 3\|y\|_{C([0,T];H^1_0)}^4 e_1^2,
\end{align*}
which yields the estimate
$\ell(e_1,y^{(1)}) \le \sqrt{3}\|y^{(1)}\|_{C([0,T];H^1_0)}^2 e_1$.
If the input $u$ is, for example, bound constrained or $L^2$ regularized
in an optimal control setting, then this implies that
$\ell(e_1,y^{(1)})$ is uniformly bounded by a multiple of $e_1$.
\end{example}

To derive an approximation of the second Newton step, we again restrict
ourselves to the case that $N$ is the superposition operator of
a polynomial such that $N \in C(W, H)$.
Taking on our comments in \cref{rem:pod_reduction_of_C},
we apply the argument of \cref{prp:cont_pod_w} to $N(y^{(1)}_\psi)$
(to the set $\mathcal{C}$). Thus there exist orthonormal vectors
$\phi^1,\ldots,\phi^m \in H$
such that the approximation error of the (pointwise a.e.) orthogonal projection
\[ e_2 \coloneqq \|N(y^{(1)}_\psi)
- \PPi_\phi N(y^{(1)}_\psi)\|_{L^2(0,T;H)}
\]
can be made arbitrarily small, where
\[ \PPi_\phi N(y^{(1)}_\psi)(t)
 = \sum_{j=1}^m (\phi^j, N(y^{(1)}_\psi)(t))_H\phi^j.
\]

We define the second simplified Newton step that is based on the
approximations $\PPi_\psi d^{(1)}$ and $\PPi_\phi N(y^{(1)}_\psi)$
\[ d^{(2)}[\psi,\phi](t) \coloneqq
\int_0^t T(t - s)\left(N(\bar{y}) - C \PPi_\psi d^{(1)} - \PPi_\phi N(y^{(1)}_\psi)\right)\,\dd s.
\]

\begin{lemma}\label{lem:approximation_d2_d2phipsi}
Let the assumptions of \cref{thm:first_newton_step_spaces}
hold. Let $N$ be the superposition operator of
a polynomial such that $N \in C(W, H)$.
Then there exists $\kappa_1 > 0$, independent of $(\phi^i)_i$
and $(\psi^j)_j$, such that
\[ \|d^{(2)} - d^{(2)}[\psi,\phi]\|_{L^2(0,T;W)}
 \le \kappa_1(e_1 + e_2 + \ell(e_1,y^{(1)})). \]
\end{lemma}
\begin{proof}
The functions $d^{(2)}$ and $d^{(2)}[\psi,\phi]$ are unique solutions
of \eqref{eq:newt2} (where the right-hand
side is changed appropriately in the case of $d^{(2)}[\psi,\phi]$).
Parabolic regularity theory gives the estimate
\begin{multline*}
\|d^{(2)} - d^{(2)}[\psi,\phi]\|_{L^2(0,T;W)} \le \\
\kappa_2\|Cd^{(1)} + N(y^{(1)}) - C \PPi_\psi d^{(1)} - \PPi_\phi N(y^{(1)}_\psi)\|_{L^2(0,T;H)}
\end{multline*}
for some $\kappa_2 > 0$. The boundedness of $C$ gives the estimate
$\|Cd^{(1)} - C d^{(1)}_\psi\|_{L^2} \le \kappa_3 e_1$,
where $\kappa_2 > 0$ is the operator norm of $C$.
The insertion of a zero and the triangle inequality yield
\[ \|N(y^{(1)}) - \PPi_\phi N(y^{(1)}_\psi)\|_{L^2(0,T;H)}
\le e_2 + \ell(e_1,y^{(1)}). \]
Thus, the claim holds with the choice
$\kappa_1 \coloneqq \kappa_2 \max\{\kappa_3,1\}$.
\end{proof}
We define
\[
   y^{(2)}_{\phi\psi}
   \coloneqq \bar{y} + \PPi_\psi d^{(1)} + d^{(2)}[\psi,\phi],
\]
where we reapply the argument of \cref{prp:cont_pod_w} and obtain a
POD approximation of $d^{(2)}[\psi,\phi]$ with basis vectors
$\{\theta^1,\ldots,\theta^k\} \subset W$. We define the approximation error
\[ e_3 \coloneqq
   \|\PPi_\theta d^{(2)}[\phi,\psi] - d^{(2)}[\phi,\psi]\|_{L^2(0,T;W)}
\]
and
\[ y^{(2)}_{\phi\psi\theta} \coloneqq \bar{y}
+ \PPi_\psi d^{(1)} 
+ \PPi_\theta d^{(2)}[\psi,\phi]. \]
We are ready to prove our main approximation result.

\begin{theorem}
Let the assumptions of \cref{thm:first_newton_step_spaces}
hold. Let $N$ be the superposition operator of
a polynomial such that $N \in C(W, H)$.
Let $y \in Y$ solve \eqref{eq:nlse}. Then
there exists $\kappa_1 > 0$, independent of $(\phi^i)_i$,
$(\psi^j)_j$, and $(\theta^\ell)_\ell$, such that
\[ \|y - y^{(2)}_{\phi\psi}\|_{L^2(0,T;W)}
   \le \|y - y^{(2)}\|_{L^2(0,T;W)} + (1 + \kappa_1)e_1 + \kappa_1(\ell(e_1,y^{(1)}) + e_2)
\]
and 
\[ \|y - y^{(2)}_{\phi\psi\theta}\|_{L^2(0,T;W)}
   \le \|y - y^{(2)}\|_{L^2(0,T;W)} + (1 + \kappa_1)e_1 + \kappa_1(\ell(e_1,y^{(1)}) + e_2) + e_3.
\]
Moreover, let $(\psi^i)_{i \in \N}$ and $(\lambda_i)_{i \in \N}$
be as in \cref{prp:cont_pod_w}, let $(\phi^j)_{j \in \N}$
and $(\mu_j)_{j \in \N}$ be an orthonormal basis of eigenvectors
and corresponding eigenvalues of a POD approximation of
$N(y^{(1)}_\psi)$, and let $(\theta^\ell)_{\ell \in \N}$
and $(\nu_\ell)_{\ell \in \N}$ be an orthonormal basis of
eigenvectors and corresponding eigenvalues of a POD approximation
of $d^{(2)}[\phi,\psi]$. Then
\begin{multline*}
\|y - y^{(2)}_{\phi\psi}\|_{L^2(0,T;W)}
\le \\
\|y - y^{(2)}\|_{L^2(0,T;W)} + (1 + \kappa_1)\sum_{i=n+1}^\infty\lambda_i  + \kappa_1 \sum_{j=m+1}^\infty\mu_j + \kappa_1 \ell(e_1,y^{(1)}),
\end{multline*}
and
\begin{multline*}
\|y - y^{(2)}_{\phi\psi\theta}\|_{L^2(0,T;W)}
\le \\
\|y - y^{(2)}\|_{L^2(0,T;W)} + (1 + \kappa_1)\sum_{i=n+1}^\infty\lambda_i  + \kappa_1 \sum_{j=m+1}^\infty\mu_j + \kappa_1 \ell(e_1,y^{(1)})
+ \sum_{\ell=k+1}^\infty \nu_\ell.
\end{multline*}
\end{theorem}
\begin{proof}
The first and second estimates follow from the estimates in
\cref{sec:approximation_newt1} and \cref{lem:approximation_d2_d2phipsi}.
The third and fourth estimates follow from \cref{prp:cont_pod_w}
and the fact that the proof of \cref{prp:cont_pod_w} can be replayed for
a POD approximation of $N(y^{(1)}_\psi)$ in the space $L^2(0,T;H)$ with basis
$(\phi^j)_{j \in \N}$ and eigenvalues $(\mu_j)_{j \in \N}$,
which gives $e_2 \le \sum_{j=m+1}^\infty\mu_j$.
An analogous argument gives $e_3 \le \sum_{\ell=k+1}^\infty\mu_\ell$.
\end{proof}

\begin{remark}\label{rem:approximation_newt2}
This means that the error of approximating $y$ with a POD approximation
of both Newton steps can be bounded by
the sum of the error of the Newton steps and four terms.
Two of them are the POD approximation errors of
the first Newton step $d^{(1)}$ and the term $N(y^{(1)}_\psi)$.
The third term relates the POD approximation error of $d^{(1)}$
to the corresponding error between $N(y^{(1)}_\psi)$ and $N(y^{(1)})$
in $L^2(0,T;H)$. As we have seen in \cref{exa:ell_e1_y1},
this error may depend on the unknown quantity $\|y^{(1)}\|_Y$, and
additional assumptions such as restrictions of the control input may
be necessary to ensure boundedness of $\|y^{(1)}\|_Y$.
The last term is the POD approximation error of
$d^{(2)}[\phi,\psi]$. For this POD approximation, the snapshots can
again be collected from impulse responses by using the characterization
developed in \cref{thm:second_newton_step_spaces}.
\end{remark}

\subsection{Discretization of the second simplified Newton step}\label{sec:discretization_second_newton_step}

We consider the $\theta$-scheme for time discretization
that we have used in \cref{sec:first_newton_step_time_discrete}
already. Again, let $0 = t_0 < \ldots < T_K = T$,
$\Delta t = t_{k+1} - t_k$ be a uniform time grid, and
let $\theta \in [\frac{1}{2},1]$.
Moreover, for $i \in \{1,\ldots,n\}$ and $j \in \{1,\ldots,m\}$
let $(u^i_k)_{k\in \{0,\ldots,K - 1\}}$ and
$(v^j_k)_{k\in \{0,\ldots,K - 1\}}$
be interval-wise constant discretizations of $u^i$ and $v^j$.

Then the discrete analogue of \eqref{eq:newt2} is
\begin{gather}\label{eq:d2k}
\begin{aligned}
  \frac{d^{(2)}_{k+1}-d^{(2)}_k}{\Delta t}-(A+C) (\theta d^{(2)}_{k+1}+(1-\theta) d^{(2)}_k)&= N(\bar y)
  - \sum_{i=1}^n u^i_k C \psi^i\\ 
  &\phantom{=}\enskip- \sum_{j=1}^m v^j_k c_j, ~ 0\le k<K,\\ 
d^{(2)}_0&= 0.
\end{aligned}
\end{gather}
Those of \eqref{eq:newt2_b} and \eqref{eq:newt2_c} are
\begin{gather}\label{eq:betaik}
\begin{aligned}
  \frac{\beta^i_{k+1}-\beta^i_k}{\Delta t}
  -(A+C) (\theta \beta^i_{k+1}+(1-\theta) \beta^i_k)&=0,\quad
0\le k <K,\\
(I+\Delta t (1-\theta) (A+C)) \beta^i_0 &= -C\psi^i
\end{aligned}
\end{gather}
and
\begin{gather}\label{eq:gammajk}
\begin{aligned}
  \frac{\gamma^j_{k+1}-\gamma^i_k}{\Delta t}
  -(A+C) (\theta \gamma^j_{k+1}+(1-\theta) \gamma^j_k)&=0,\quad
0\le k <K,\\
(I+\Delta t (1-\theta) (A+C)) \gamma^j_0 &= c_j.
\end{aligned}
\end{gather}
The  analog of \eqref{eq:newt2}
with $d^{(2)} = 0$ and $N(y^{(1)}) = 0$ is
\begin{gather}\label{eq:rk}
\begin{aligned}
  \frac{r_{k+1}-r_k}{\Delta t}-(A+C) (\theta r_{k+1}+(1-\theta) r_k)
  &= N(\bar y),~ 0\le k<K,\\ 
  r_0&= 0.
\end{aligned}
\end{gather}
We obtain the following discrete convolution formula for the
$\theta$-scheme.

\begin{proposition}\label{prp:newt2_conv_theta}
Consider the first simplified Newton step $(d^{(1)}_k)$ for the
$\theta$-scheme \eqref{eq:d1k}.
Let $(\beta^i_k)$ and $(\gamma^j_k)$ be the solutions of
\eqref{eq:betaik} and \eqref{eq:gammajk}, respectively.
Then $(d^{(2)}_k)$ can be represented by the discrete convolution
formula
\begin{align*}
	d^{(2)}_{k} = r_k
	+ \Delta t \sum_{i=1}^n \sum_{\ell=0}^{k-1}\beta^i_{k-\ell}u^i_{\ell}
	+ \Delta t \sum_{j=1}^n \sum_{\ell=0}^{k-1}\gamma^j_{k-\ell}v^j_{\ell},
\end{align*}
where we use the convention that the sum vanishes for $k=0$.
Consequently,
\[ d^{(2)}_{k} - r_k \in \spn\left\{ \bigcup_{i=1}^n \{\beta^i_1,\ldots,\beta^i_k\} \cup  \bigcup_{j=1}^m \{\gamma^j_1,\ldots,\gamma^j_k\} \right\}. \]
\end{proposition}
\begin{proof}
We define $e^i_k \coloneqq \Delta t \sum_{\ell=0}^{k-1} \beta^i_{k - \ell} u^i_{\ell}$ and $f^j_k \coloneqq \Delta t \sum_{\ell=0}^{k-1} \gamma^j_{k - \ell} v^j_{\ell}$.
Then $e^i_0=0$, $f^j_0 = 0$, and from \eqref{eq:betaik}
and \eqref{eq:gammajk} we obtain---analogously to \cref{prp:conv_theta}---that
\begin{align*}
\frac{e^i_{k+1}-e^i_k}{\Delta t} &= \sum_{\ell=0}^{k}(\beta^i_{k+1-\ell} - \beta^i_{k-\ell}) u^i_{\ell} + \beta^i_{0}u^i_k 
= (A + C)(\theta e^i_{k+1}-(1-\theta)e^i_k) - Cb_iu^i_k
\end{align*}
and
\begin{align*}
\frac{f^j_{k+1}-f^j_k}{\Delta t} &= \sum_{\ell=0}^{k}(\gamma^j_{k+1-\ell} - \gamma^j_{k-\ell}) v^j_{\ell} + \beta^j_{0}v^j_k
= (A + C)(\theta f^j_{k+1}-(1-\theta)f^j_k) + c_j v^j_k.
\end{align*}
By superposition $r_k + \sum_{i=1}^n e^i_k + \sum_{j=1}^m f^j_k$
satisfies \eqref{eq:d2k} as asserted. The last claim follows
by inspection.
\end{proof}

\section{Galerkin ansatz}\label{sec:galerkin}

We derive error estimates of a Galerkin ansatz with the
\ac{POD} basis vectors to approximate the space $W$.
To this end, we consider the bilinear forms $a : W \times W \to \R$
and $c : W \times W \to \R$ that arise from the linear operators
$A$ and $C$ in the general setting of \cref{sec:general_framework}.

\subsection{Error bound for Newton steps on \ac{POD} model}

We consider $W_\psi \coloneqq \spn\{\psi^1,\ldots,\psi^n\} \subset W$.
Let $d^{(1)}_{\psi}$ solve \eqref{eq:newt1} on $W_\psi$; that is,
\begin{gather}\label{eq:d1_psi}
\begin{aligned}
(\partial_t d^{(1)}_\psi, v_\psi)_H
+ a(d^{(1)}_\psi, v_\psi)
+ c(d^{(1)}_\psi, v_\psi) - (Fu, v_\psi)_H &= 0, \\
(d^{(1)}_\psi(0) - \bar{y} - y_0, v_\psi)_H &= 0
\end{aligned}\tag{N1}
\end{gather}
for all $v_\psi \in W_\psi$.
Moreover, we consider the subspace
$W_\theta \coloneqq \spn\{\theta^1,\ldots,\theta^k\}$.
Let $d^{(2)}_{\theta}$ solve the second simplified Newton step
\eqref{eq:newt2} on $W_{\theta}$; that is,
\begin{gather}\label{eq:d2_theta}
\begin{aligned}
(\partial_t d^{(2)}_\theta, v_\theta)_H
+ a(d^{(2)}_\theta, v_\theta)
+ c(d^{(2)}_\theta, v_\theta) - (r, v_\theta)_H - c(d^{(1)}_\psi, v_\theta) &= 0, \\
(d^{(2)}_\theta(0), v_\theta)_H &= 0
\end{aligned}\tag{N2}
\end{gather}
for all $v_\theta \in W_\theta$, where $r = N(\bar{y}) - \PPi_\phi N(y^{(1)}_\psi)$.

\begin{theorem}
Let $a + c$ be a coercive bilinear form on $W$.
Let $(\psi^i)_i$, $(\phi^j)_j$, and $(\theta^\ell)_\ell$
be as in \cref{sec:approximation_newt1,sec:approximation_newt2}.
Then there exist $\kappa_2, \kappa_3 > 0$ such that 
\begin{multline*}
\|y - (\bar{y} + d^{(1)}_\psi + d^{(2)}_\theta)\|_{L^2(0,T;W)}\le \\
   \|y - y^{(2)}\|_{L^2(0,T;W)}
+ \kappa_2(1 + \kappa_1\kappa_3) e_1 + \kappa_1\kappa_3(\ell(e_1,y^{(1)}) + e_2) + 
\kappa_3 e_3.
\end{multline*}
\end{theorem}
\begin{proof}
The coercivity of $a + c$ allows us to obtain the error bounds
\begin{align*}
\|d^{(1)} - d^{(1)}_\psi\|_{L^2(0,T;W)}
&\le \kappa_2 \|d^{(1)} - \PPi_\psi d^{(1)}\|_{L^2(0,T;W)},\text{ and}\\
\|d^{(2)}[\phi,\psi] - d^{(2)}_\theta\|_{L^2(0,T;W)}
&\le \kappa_3 \|d^{(2)}[\phi,\psi] - \PPi_\theta d^{(2)}[\phi,\psi]\|_{L^2(0,T;W)}
\end{align*}
for $\kappa_2, \kappa_3 > 0$ by following, for example, the proof of
Theorem 2.3 in \cite{bernardi1985conforming}. Then the claim follows
after combining these error bounds with the triangle inequality
\[\|d^{(2)} - d^{(2)}_\theta\|_{L^2(0,T;W)} \le 
\|d^{(2)} - d^{(2)}[\phi,\psi]\|_{L^2(0,T;W)}
+\|d^{(2)}[\phi,\psi] - d^{(2)}_\theta\|_{L^2(0,T;W)},\]
the bound from \cref{lem:approximation_d2_d2phipsi},
and the \ac{POD} approximation errors
$e_1$ and $e_3$.
\end{proof}

\subsection{Galerkin approximation error for the nonlinear equation}

Let $a : W \times W \to \R$ be a continuous and coercive
bilinear form, and let $N : W \to H$ be a polynomial.
We consider the variational formulations of \eqref{eq:nlse} on $W$,
\begin{gather}\label{eq:q}
\left(\partial_t y, v\right)_H
+ a(y, v)
+ (N(y), v)_H - (Fu, v)_H = 0,\quad (y(0) - y_0, v) = 0
\tag{\text{Q}}
\end{gather}
for all $v \in W$, and on $W_\varrho \coloneqq \spn \{\varrho^{1},\ldots,\varrho^{k}\} \subset W$,
\begin{gather}\label{eq:q_varrho}
\left(\partial_t y_\varrho, v_\varrho\right)_H
+ a(y_\varrho, v_\varrho)
+ (N(y_\varrho), v_\varrho)_H - (Fu, v_\varrho)_H = 0,\quad
(y_\varrho(0) - y_0, v_\varrho) = 0
\tag{\text{Q$_\varrho$}}
\end{gather}
for all $v_\rho \in W_\rho$. Let $y$ solve \eqref{eq:q}, and let
$y_\rho$ solve \eqref{eq:q_varrho}. We estimate
$\|y_\varrho - \PPi_\varrho y\|_{H}$ below.

\begin{theorem}
Let the nonlinearity $N$ satisfy the estimate
\begin{gather}\label{eq:N_lbar_estimate}
\|N(y)-N(y_\varrho)\|_{H} \le \bar{\ell}(\|y\|_W + \|y_\varrho\|_W) \|y - y_\varrho\|_{W}
\end{gather}
for some monotone function $\bar{\ell}: [0,\infty) \to [0,\infty)$.
Then it holds for $t \in [0,T]$ that
\[ 
\|(y_\varrho - \PPi_\varrho y)(t)\|_H^2\le
c_1 c_2 \left(\|(y_\varrho - \PPi_\varrho y)(0)\|_H^2 + 
\|y - \PPi_\varrho y\|_{L^2(0,t;W)}^2\right), 
\]
where $c_1 > 0$ is an independent constant and
$c_2 = e^{2T \bar{l}(\|y\|_{L^\infty(0,T;W)} + \|y_\varrho\|_{L^\infty(0,T;W)})^2}$.
\end{theorem}
\begin{proof}
We follow the ideas of \cite[Thm\,2.3]{bernardi1985conforming} and observe
that $(\partial_t (y - \PPi_\varrho y), y_\varrho - \PPi_\varrho y)_H = 0$.
Combining this with the choice $y_\varrho - \PPi_\varrho y$ for the test functions
in \eqref{eq:q_varrho} and \eqref{eq:q} and following the steps in
\cite[Thm\,2.3]{bernardi1985conforming}, we have that
\begin{multline*}
\frac{1}{2} \partial_t (y_\varrho - \PPi_\varrho y, y_\varrho - \PPi_\varrho y)
+ \frac{1}{2} a(y_\varrho - \PPi_\varrho y, y_\varrho - \PPi_\varrho y)\le \\
\frac{1}{2} a(y - \PPi_\varrho y, y - \PPi_\varrho y)
+ (N(y) - N(y_\varrho), y_\varrho-\PPi_\varrho y)_H.
\end{multline*}
The estimate \eqref{eq:N_lbar_estimate} and the triangle inequality yield
\[ \|N(y) - N(y_\varrho)\|_H
 \le \bar{\ell}(\|y\|_W + \|y_\varrho\|_W)
(\|y - \PPi_\varrho y\|_W + \|y_\varrho - \PPi_\varrho y\|_W).
\]
We apply the Cauchy--Schwarz inequality to
$(N(y) - N(y_\varrho), y_\varrho-\PPi_\varrho y)_H$,
insert the estimate above, and apply the inequality
$ab \le \frac{1}{2}a^2 + \frac{1}{2}b^2$ suitably to obtain
\begin{multline*}
(N(y) - N(y_\varrho), y_\varrho-\PPi_\varrho y)_H \le
\frac{1}{4}a(y - \PPi_\rho y,y - \PPi_\rho y) \\
+ \frac{1}{4}a(y_\varrho - \PPi_\rho y,y_\varrho - \PPi_\rho y)
+ \frac{2}{\alpha^2}\bar{l}(\|y\|_W + \|y_\varrho\|_W)^2\|y_\varrho-\PPi_\varrho y\|_H^2,
\end{multline*}
where $a(v,v) \ge \alpha \|v\|_W^2$ by coercivity.
Then the bilinearity and coercivity of $a$ yield
\begin{multline*}
\frac{1}{2} \partial_t (y_\varrho - \PPi_\varrho y, y_\varrho - \PPi_\varrho y)_H
+ \frac{\alpha}{4} \|y_\varrho - \PPi_\varrho y\|_W^2 \le \\
\frac{3}{4} \|y - \PPi_\varrho y\|_W^2
+ \frac{2}{\alpha^2}\bar{l}(\|y\|_W + \|y_\varrho\|_W)^2\|y_\varrho-\PPi_\varrho y\|_H^2.
\end{multline*}
Making the dependency on $t$ explicit and rearranging, we obtain 
\begin{multline*}
\partial_t \frac{1}{2}\|(y_\varrho - \PPi_\varrho y)(t)\|_H^2 \le
\frac{3}{4} \|(y - \PPi_\varrho y)(t)\|_W^2\\
+ \frac{2}{\alpha^2}\bar{l}(\|y(t)\|_W + \|y_\varrho(t)\|_W)^2\|(y_\varrho-\PPi_\varrho y)(t)\|_H^2
- \partial_t \frac{\alpha}{4}\|y_\varrho - \PPi_\varrho y\|_{L^2(0,t;W)}^2.
\end{multline*}
We scale by $2$ and apply the Gronwall lemma to obtain
\begin{multline*}
\|(y_\varrho - \PPi_\varrho y)(t)\|_H^2\le\\
c_2 \left(\|(y_\varrho - \PPi_\varrho y)(0)\|_H^2
+ \frac{3}{2}\int_0^t \|(y - \PPi_\varrho y)(s)\|_W^2 \dd s
	- \frac{\alpha}{2}\|(y_\varrho - \PPi_\varrho y)(t)\|_W^2 .
\right) 
\end{multline*}
We use the estimate $\|y_\varrho - \PPi_\varrho y\|_H^2 \le \beta \|y_\varrho - \PPi_\varrho y\|_W^2$ for some $\beta > 0$
and the fact that $c_2 \ge 1$ to deduce that there
exists $c_1 > 0$ such that
\begin{gather*}
\|(y_\varrho - \PPi_\varrho y)(t)\|_H^2\le
c_1 c_2
\left(\|(y_\varrho - \PPi_\varrho y)(0)\|_H^2 + 
\|y - \PPi_\varrho y\|_{L^2(0,t;W)}^2\right).
\end{gather*}
\end{proof}

\section{Augmented \acs{POD} basis computation}\label{sec:augmented_pod_basis_computation}

Having established the theoretical framework above, we argue for the following 
augmentation of the common \ac{POD} basis computation procedure.
We compute  $v$ and collect impulse response snapshots $w(t)$
using $F$ as initial value. Then, we reduce the collected set with
\ac{POD} and obtain a reduced basis $\mathcal{B}^{(1)}$ of $d^{(1)}$,
\cf \cref{thm:first_newton_step_spaces}.
We compute a basis of a linear subspace $\mathcal{C}$,
in which $N(y^{(1)}_{\mathcal{B}^{(1)}})$ takes its values,
\cf \cref{sec:second_newton_step}.
This step depends on the nonlinearity $N$. For $N(y) = by^3$ in our
guiding example, we have 
\[ \mathcal{C} = \spn\left\{(\bar{y} + \psi^i)(\bar{y} + \psi^j)(\bar{y} + \psi^k) \,\vert\, \forall \text{ combinations } i,j,k \right\}
\]
for $\mathcal{B}^{(1)} = \{\psi^1,\ldots,\psi^n\}$.
Now, we compute impulse responses for the right-hand side of the second
Newton step given in \cref{lem:second_newton_step_formula} 
by means of impulse response snapshots, \cf \cref{thm:second_newton_step_spaces}.

After collecting and reducing the snapshots, we obtain the basis
$\mathcal{B}^{(2)}$ of $d^{(2)}$. Because
$y^{(2)} = \bar{y} + d^{(1)} + d^{(2)}$,
we can compute a basis and reduced \ac{FEM} operators for the
second Newton iterate $y^{(2)}$ by applying \ac{POD} to
the set $\{ \bar{y} \} \cup \mathcal{B}^{(1)} \cup \mathcal{B}^{(2)}$.
We summarize this procedure in \cref{alg:two_step_pod_computation}.

\begin{algorithm}
\caption{Two-step Newton-based \ac{POD} Computation}\label{alg:two_step_pod_computation}  
\begin{algorithmic}[1]
  \REQUIRE{\ac{IBVP} solution operator $\solve$, \ac{POD} basis computation $\podop$} 
  \REQUIRE{Linearization point $\bar{y}$}
  \STATE{$\hphantom{(w_k)_k}\mathllap{(v_k)_k} \gets \solve$ \eqref{eq:vk}}  
  \STATE{$(w_k)_k \gets \solve$ \eqref{eq:wk}}  
  \STATE{$\hphantom{(w_k)_k}\mathllap{\mathcal{B}^{(1)}} \gets \podop( (v_k)_k \cup (w_k)_k )$}
  \STATE{$\{\phi^1,\ldots,\phi^m\} \gets $ \texttt{compute basis of $N(y^{(1)})$ from $\mathcal{B}^{(1)}$}}
  \STATE{$\hphantom{(w_k)_k}\mathllap{(r_k)_k} \gets \solve$ \eqref{eq:rk}}  
  \FOR{$i = 1$ \TO $n$} 
  	\STATE{$(\beta^i_k)_k \gets \solve$ \eqref{eq:betaik}}  
  \ENDFOR  
  \FOR{$j = 1$ \TO $m$} 
  	\STATE{$(\gamma^j_k)_k \gets \solve$ \eqref{eq:gammajk}}
  \ENDFOR
  \STATE{$\hphantom{(w_k)_k}\mathllap{\mathcal{B}^{(2)}} \gets \podop\left( (r_k)_k \cup \bigcup_{i=1}^n (\beta^{i}_k)_k \cup \bigcup_{j=1}^n (\gamma^j_{k})_k \right)$}
  \STATE{$\hphantom{(w_k)_k}\mathllap{\mathcal{B}^{(12)}} \gets \podop\left( \{ \bar{y} \} \cup \mathcal{B}^{(1)} \cup \mathcal{B}^{(2)} \right)$}
  \RETURN{$\mathcal{B}^{(12)}$}
\end{algorithmic}  
\end{algorithm} 

\section{Computational results}\label{sec:computational_results}
We demonstrate our findings by means of a numerical implementation of the guiding 
example from \cref{sec:guiding_example}.

\subsection{Setup}
We have chosen $a = 0.01$ and $b = 3$ as parameters for  the \ac{PDE} and its linearizations. Regarding the time domain, 
we have used an equidistant grid consisting of $65$ intervals. The time stepping has been realized with the help of the backward 
Euler method. Regarding the spatial domain, we have used finite elements of quadratic order on a triangulation of an L-shaped 
domain.
For the linearization point $(\bar{y}, \bar{u})$, we have set $\bar{u} \equiv 2$ as well as $\partial_t \bar{y} \equiv 0$ and computed 
the resulting $\bar{y}$ to solve \eqref{eq:ex_sl}.
%
%
Regarding the error or difference computations between state vectors, we note that we have always used the $H^1$-norm 
for the spatial domain. The same applies for the \ac{POD} computations.

\subsection{Approximation with two simplified Newton steps}\label{sec:approx_two_newton_steps}
We compare the solution $y$ of \eqref{eq:ex_sl} to
$y^{(1)} = \bar{y} + d^{(1)}$ and
$y^{(2)} = \bar{y} + d^{(1)} + d^{(2)}$
for a given test control $u$, which is displayed in \cref{fig:u_test},
and given $\bar{y} \equiv y_0$. We have computed $y^{(1)}$ and $y^{(2)}$ with the help of the linearizations 
of \eqref{eq:ex_sl} described in \cref{sec:evolution_equations}. The relative difference between $y^{(2)}$ and $y$ is more than one order of
magnitude smaller than the relative difference between $y^{(1)}$ and $y$. 
We have computed $\mathcal{B}^{(1)}$ and $\mathcal{B}^{(12)}$ 
by means of \cref{alg:two_step_pod_computation}. Consequently, \eqref{eq:ex_sl} has been solved by using the reduced spaces, 
\ie reduced versions of the operators, yielding solutions $y_{\mathcal{B}^{(1)}}$, $y_{\mathcal{B}^{(12)}}$. We are interested in their ability 
to approximate $y$. We observe that the relative approximation error of $y_{\mathcal{B}^{(12)}}$ is two orders of magnitude smaller than 
that of $y_{\mathcal{B}^{(1)}}$. 

We note that the dimension of the discrete state vectors $y$ using
the \ac{FEM} matrices was $2037$, $14$ for $y_{\mathcal{B}^{(1)}}$
and $1168$ for $y_{\mathcal{B}^{(12)}}$. The exact results of
these four
computations are given in \cref{tbl:y_y1_y2_relation}. The high
number of basis vectors in $\mathcal{B}^{(12)}$ is due to the fact
that we have included every basis vector of $\mathcal{B}^{(2)}$ except
for those with a singular value smaller than $10^{-8}$,
the cutoff value of the \ac{SVD}, into $\mathcal{B}^{(12)}$.
It is interesting what happens when we do not use all of
them and drop those corresponding very small singular values. This
situation is investigated in the context of an \ac{OCP} in the next subsection.
\begin{table}
\caption{Relative approximation error between $y$ and the Newton step approximations $y^{(1)}$, and $y^{(2)}$ (\ac{FEM} model)
as well as the POD approximations $y_{\mathcal{B}^{(1)}}$ and $y_{\mathcal{B}^{(12)}}$.}\label{tbl:y_y1_y2_relation}
\begin{center}
\begin{tabular}{cccc}
	\hline 
	$\frac{\|y - y^{(1)}\|}{\|y\|}$ & $\frac{\|y - y^{(2)}\|}{\|y\|}$ & $\frac{\|y - y_{\mathcal{B}^{(1)}}\|}{\|y\|}$ & $\frac{\|y - y_{\mathcal{B}^{(12)}}\|}{\|y\|}$ \\ \hline 
	\num{8.8660e-04} & \num{6.3124e-05} & \num{1.6310e-04} & \num{1.3320e-06} \\ \hline
\end{tabular}
\end{center}
\end{table}
%

\subsection{Application to an \acf{OCP}}\label{sec:application_to_ocp}

We have solved the following a tracking-type \ac{OCP} with given desired
state $y_d$ and Tikhonov regularization parameter $\gamma = 10^{-7}$:
\begin{gather*}
	\min_{y,u} \frac{1}{2}\|y - y_d\|_{Z}^2 + \frac{\gamma}{2}\|u\|_{L^2}^2\quad
	\text{s.t.}\quad\partial_t y - a\Delta y + by^3 = Fu,\enskip
	y|_{\partial \Omega} = 0.
\end{gather*}
A reduced objective approach has been chosen to obtain an unconstrained \ac{OCP}.
The optimization routine has been initialized with $u \equiv 0$.
The target state $y_d$ is the solution of the state equation for
the control input visualized in \cref{fig:u_ref}.
We have solved the \ac{IBVP}  with \ac{FEM}, $\mathcal{B}^{(1)}$, and
$\mathcal{B}^{(12)}$. Regarding $\mathcal{B}^{(12)}$,
we have run the computations for different sizes $B_2 = |\mathcal{B}^{(2)}|$.
Specifically, we have successively increased the number of basis vectors in
$\mathcal{B}^{(2)}$ following a descending order of the corresponding
singular values. The experiment has been run on two spatial grids
with different mesh sizes.

\begin{figure}
\begin{subfigure}[c]{0.45\linewidth}
\includegraphics[height=4.25cm]{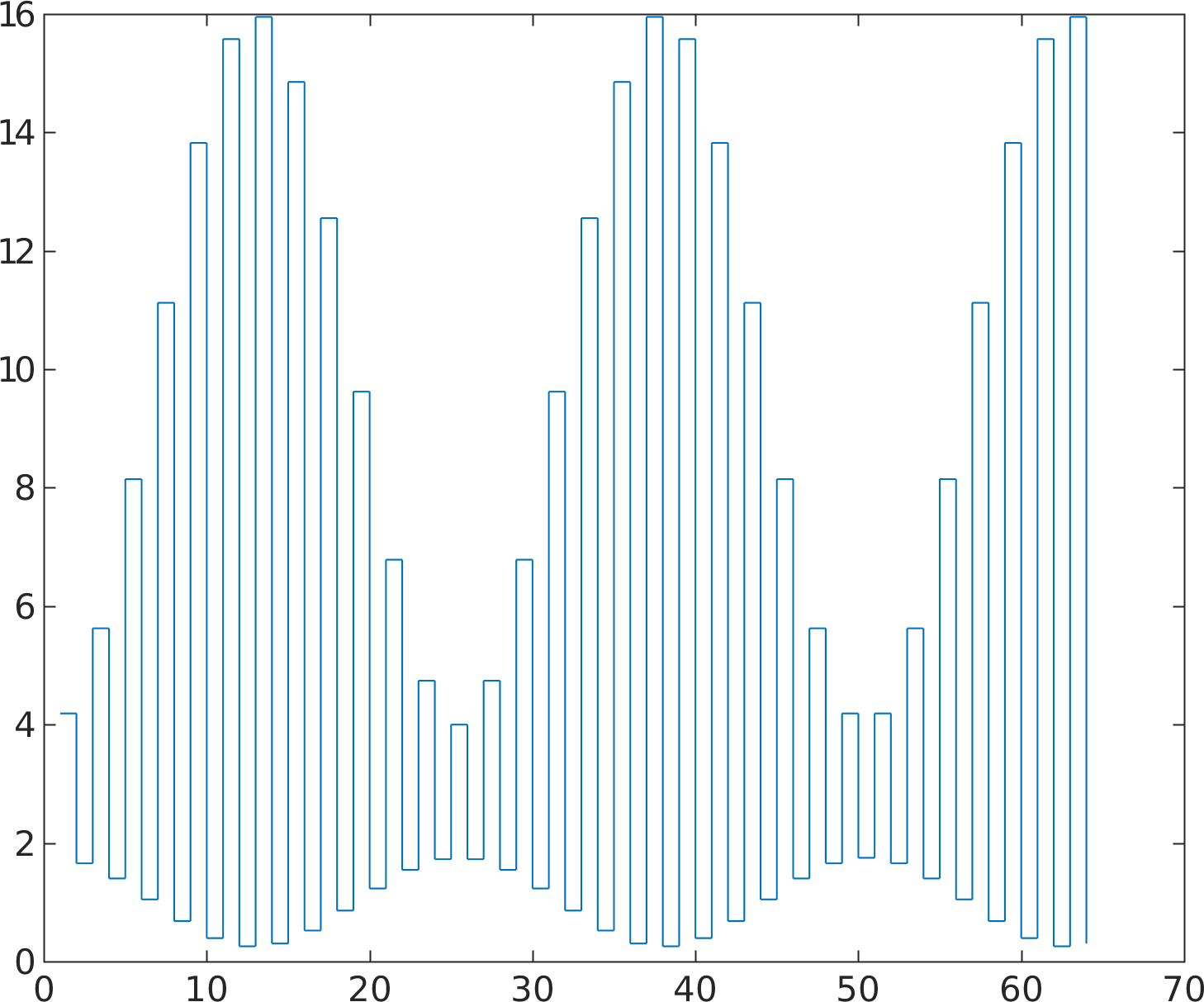}
\caption{Reference control to compute $y_d$
in \cref{sec:application_to_ocp}.}\label{fig:u_ref}
\end{subfigure}
\hfill
\begin{subfigure}[c]{0.45\linewidth}
\includegraphics[height=4.25cm]{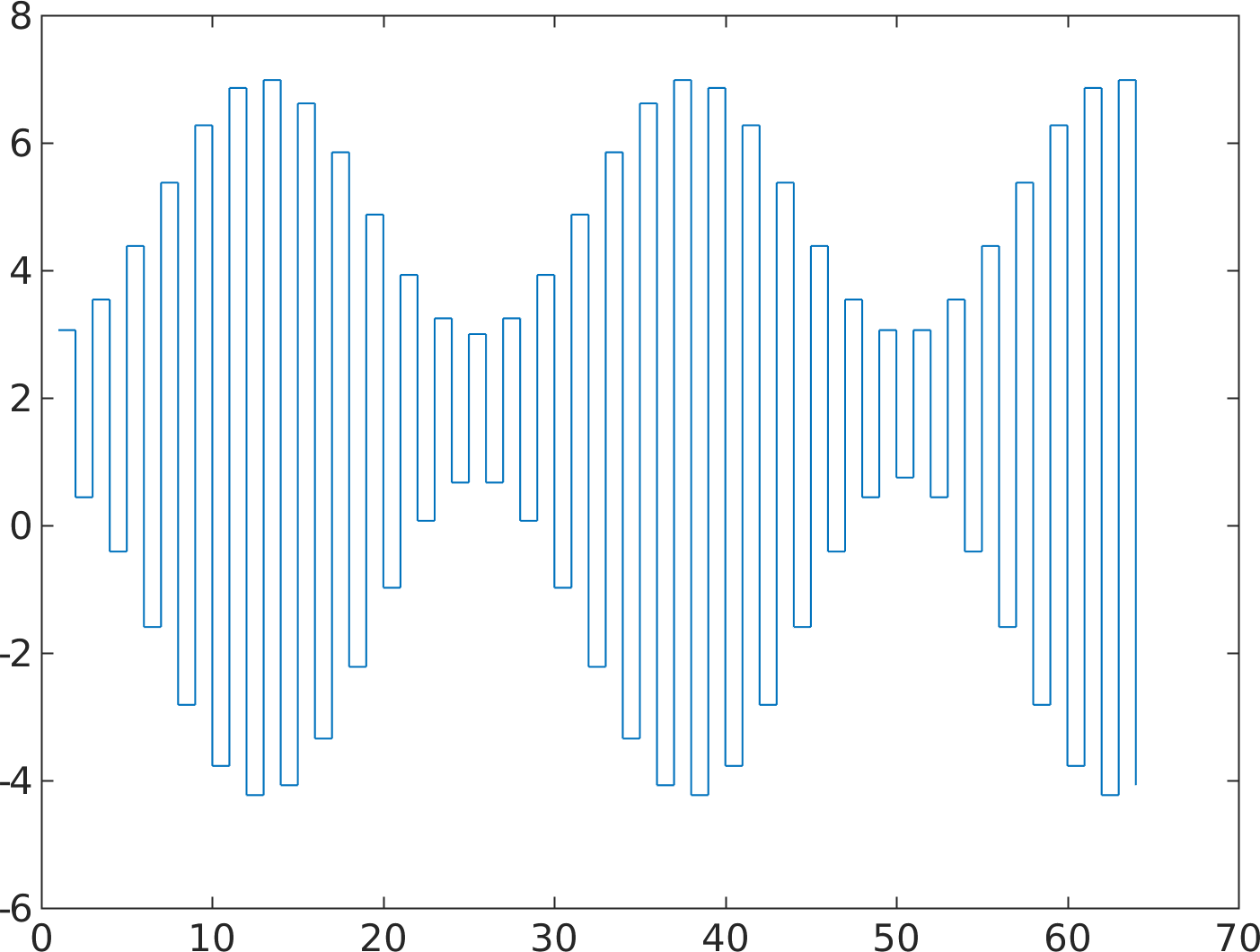}
\caption{Test control input for the experiment in
\cref{sec:approx_two_newton_steps}.}\label{fig:u_test}
\end{subfigure}
\caption{Control inputs for computational test cases.}
\end{figure}

On the coarse grid, the state vector of the \ac{FEM} discretization has $2,469$
entries while the state vector of the one-step POD $\mathfrak{B}^{(1)}$
has $12$ entries. On the fine grid, the state vector of \ac{FEM} 
discretization has $5,597$ entries while the state vector of the discretization using the one-step POD $\mathcal{B}^{(1)}$ has $13$ entries. In both cases,
vectors from $\mathcal{B}^{(2)}$ were included in
$\mathcal{B}^{(12)}$ until the corresponding singular value fell below
$10^{-8}$. For both grids, the additional basis vectors yield more
accurate optimized objective values compared with the \ac{FEM} discretization.
Adding $10$ basis vectors yields a drop of the relative error in the objective 
value from $\num{2.56e-2}$ to $\num{1.25e-4}$ while the computation
time increases from $\SI[scientific-notation=false]{229}{\second}$ to
$\SI[scientific-notation=false]{306}{\second}$, compared with
$\SI[scientific-notation=false]{9733}{\second}$ 
for the \ac{FEM} solution. The number of optimization iterations stays
almost constant: $44$ iterations are used on the \ac{FEM} model,
$43$ on the $\mathcal{B}^{(1)}$ model,
and $44$ on all $\mathcal{B}^{(12)}$ models. 
Similarly, on the fine grid, adding $10$ basis vectors resulted in
a drop in the relative error in the objective value from
$\num{2.34e-2}$ to $\num{7.22e-5}$ while the computation time increased
from $\SI[scientific-notation=false]{428}{\second}$ to 
$\SI[scientific-notation=false]{555}{\second}$, compared with
$\SI[scientific-notation=false]{41739}{\second}$ for the \ac{FEM}
solution. In all cases, the optimization consumes $40$
iterations.

Two figures illustrate our results. \Cref{fig:b12_size_vs_relobjerror}
shows how the relative objective error decreases for
for an increasing basis $\mathcal{B}^{(12)}$.
\begin{figure}
    \centering
\begin{tikzpicture}[scale=0.65]
\begin{axis}[ymode=log, ylabel=$e$, ylabel style={yshift=-0.25cm,rotate=-90}, xlabel=\# basis vectors]
\addplot[blue,mark=*,select coords between index={2}{99}] table [x=size,y=relobj,col sep=comma] {ie_0.075.csv};
\addplot[green,mark=*,select coords between index={1}{1}] table [x=size,y=relobj,col sep=comma] {ie_0.075.csv};
\end{axis}
\end{tikzpicture}\hspace{1cm}
\begin{tikzpicture}[scale=0.65]
\begin{axis}[ymode=log, ylabel=$e$, ylabel style={yshift=-0.25cm,rotate=-90}, xlabel=\# basis vectors]
\addplot[blue,mark=*,select coords between index={2}{100}] table [x=size,y=relobj,col sep=comma] {ie_0.05.csv};
\addplot[green,mark=*,select coords between index={1}{1}] table [x=size,y=relobj,col sep=comma] {ie_0.05.csv};
\end{axis}
\end{tikzpicture}        
\caption{Relative objective error (vertical axes) against size of
$\mathcal{B}^{(12)}$ (horizontal axes) for coarse (left) and
fine (right) grid. The relative objective error for the
$\mathcal{B}^{(1)}$ model is marked in
green.}\label{fig:b12_size_vs_relobjerror}
\end{figure}
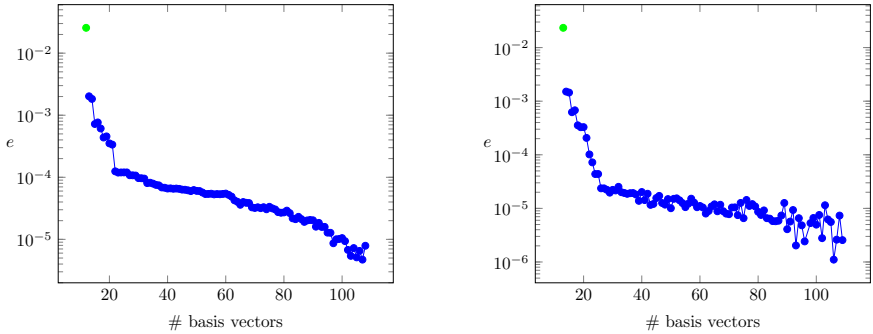
\Cref{fig:b12_size_vs_timeconsumption} shows the running time
of the \ac{OCP} solves with the $\mathcal{B}^{(12)}$ model
for an increasing number of basis vectors.
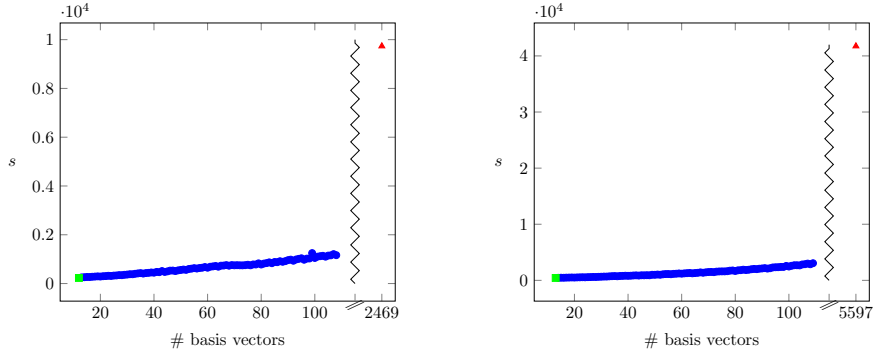
\begin{figure}
    \centering
\begin{tikzpicture}[scale=0.65]
\begin{axis}[xmin=5,
             xmax=130,
             xticklabels={20,40,60,80,100,2469},
             xtick=      {20,40,60,80,100,125},
			 extra x ticks={115},
			 extra x tick style={grid=none, tick label style={xshift=0cm,yshift=.30cm, rotate=-45}},
			 extra x tick label={\color{black}{/\!\!/}},
			 ylabel=$s$, ylabel style={yshift=-0.25cm,rotate=-90},
			 xlabel=\# basis vectors]
\addplot[blue,mark=*,select coords between index={2}{99}] table [x=size,y=time,col sep=comma] {ie_0.075.csv};
\addplot[green,mark=square*,select coords between index={1}{1}] table [x=size,y=time,col sep=comma] {ie_0.075.csv};
\addplot[red,mark=triangle*,select coords between index={0}{0}] table [x index=1,y=time,col sep=comma] {ie_0.075.csv};
\draw[black] decorate [decoration={zigzag}] {(axis cs:115,0) -- (axis cs:115,1e4)};
\end{axis}
\end{tikzpicture}\hspace{1cm}
\begin{tikzpicture}[scale=0.65]
\begin{axis}[xmin=5,
             xmax=130,
             xticklabels={20,40,60,80,100,5597},
             xtick=      {20,40,60,80,100,125},
			 extra x ticks={115},
			 extra x tick style={grid=none, tick label style={xshift=0cm,yshift=.30cm, rotate=-45}},
			 extra x tick label={\color{black}{/\!\!/}},
			 ylabel=$s$, ylabel style={yshift=-0.5cm,rotate=-90},
			 xlabel=\# basis vectors]
\addplot[blue,mark=*,select coords between index={2}{100}] table [x=size,y=time,col sep=comma] {ie_0.05.csv};
\addplot[green,mark=square*,select coords between index={1}{1}] table [x=size,y=time,col sep=comma] {ie_0.05.csv};
\addplot[red,mark=triangle*,select coords between index={0}{0}] table [x index=1,y=time,col sep=comma] {ie_0.05.csv};
\draw[black] decorate [decoration={zigzag}] {(axis cs:115,0) -- (axis cs:115,4.2e4)};
\end{axis}
\end{tikzpicture}        
\caption{Running time of the \ac{OCP} (vertical axes) against the number of
basis vectors (horizontal axes)
for coarse (left) and fine (right) grid 
for $\mathcal{B}^{(1)}$ (green squares), increasing
$\mathcal{B}^{(12)}$ (blue dots), and \ac{FEM} (red triangle) models.}\label{fig:b12_size_vs_timeconsumption}
\end{figure}

\section{Application in mixed-integer optimal control}\label{sec:miocp}

We briefly outline how the presented model reduction can be help to solve relaxations of \acp{MIOCP}.
Employing Sager's convexification technique \cite{sager2005numerical,sager2012integer} to control problems constrained by semilinear evolution equations
with discrete-valued control inputs, one obtains state equations of the form
\[ \partial_t y - Ay = \sum_{i=1}^M \omega_i f_i(y),\quad y(0) = y_0 \]
with $\omega_i \in L^\infty((0,T), \mathbb{R}^M)$ and $\omega_i(t) \in \{0,1\}^M$ and $\sum_{i=1}^M\omega_i(t) = 1$ for a.a.\ $t \in [0,T]$;
see \cite{hante2013relaxation,manns2019improved}. The $\omega_i$ may be regarded as activations of the different right-hand sides $f_i$.
Following the ideas in \cite{sager2012integer}, one can approach the \ac{OCP} by first solving a relaxation in which the constraint $\omega_i(t) \in \{0,1\}^M$
is relaxed to $\omega_i(t) \in [0,1]^M$ and then computing a binary-valued approximation of the relaxed activation, a procedure that is known as
\emph{combinatorial integral decomposition}; see \cite{jung2015the}.
Using the proposed method for snapshot generation, one can obtain improved reduced models for the semilinear equations
\[ \partial_t y - Ay = \omega_i f_i(y),\quad y(0) = y_0 \]
for $i \in \{1,\ldots,M\}$ and in particular reduced bases for the terms
\begin{gather}\label{eq:conv_term}
	\int_0^t S(t - s)f_i(y(s))\omega_i(s) \dd s,\quad i \in \{1,\ldots,M\}
\end{gather}
if $(S(t))_{t \ge 0}$ denotes the semigroup generated by $A$. By the variation of constants formula, we have
\[ y(t) = S(t)y_0 + \sum_{i=1}^M \int_0^t S(t - s)f_i(y(s))\omega_i(s) \dd s. \]
Consequently, by combining the bases for the terms in \eqref{eq:conv_term} using a POD computation as in \cref{alg:two_step_pod_computation}, one obtains
an efficient approximation of the solution operator of the semilinear equation in the relaxed problem.

If many relaxations have to be solved, for example in a branch-and-bound procedure, high-quality surrogate models are even more important.
In \cite{bachmann2019pod}, \ac{POD} models are used for a linear parabolic equation in a branch-and-bound procedure. We envision efficient treatment of semilinear equations in this context using the proposed method.

\section{Conclusion}\label{sec:conclusion}

We have developed an algorithm to compute \ac{POD} models for a class of 
semilinear evolution equations using the approximation properties of
simplified Newton steps on the state equation.
The computational results validate the theoretical findings. Furthermore,
we have solved a tracking-type \ac{OCP} constrained by a semilinear \ac{PDE}
from the investigated class on an FEM model, the one-step \ac{POD} model and
a sequence of increasingly augmented \ac{POD} models. A moderate number of additional
basis vectors improves the approximation of the optimization on the
\ac{FEM} model significantly compared with the one-step \ac{POD} model. 

Thus, if one is willing to spend the expensive \emph{offline phase} for snapshot
generation in \cref{alg:two_step_pod_computation}, 
for example because many similar \acp{OCP} have to be solved in an \ac{MPC} context or to solve relaxations of \acp{MIOCP}, one can  trade in a
moderate loss in speed-up of the reduced model for a much better capture of the result.
For example, in our computational setup on the fine grid,
we have achieved an improvement of the relative objective error by a factor of $200$ at the cost of approximately halving the speed-up when including
$30$ additional basis vectors of $\mathfrak{B}^{(2)}$ into $\mathfrak{B}^{(12)}$.

\appendix
\section{The continuous embedding $\mathcal{H} \hookrightarrow C([0,T],H^1_0(\Omega))$}\label{app:embedding}

For existence of solutions of the semilinear equation \eqref{eq:ex_sl}
in \cref{prp:ex_setting}, we refer to
\cite[Prop.\ 5.1]{barbu2010nonlinear}. Considering the results therein,
a regularity of the solution 
in the space $C([0,T];H^1_0(\Omega))$ seems to be out of reach.
Moreover, the application of the vector-valued embedding theorem
\cite[Thm 8.60]{leoni2017first} with the choices
$Y = H^2(\Omega) \cap H^1_0(\Omega)$ and $H = H^1_0(\Omega)$
seems to require a simultaneous identification of both Hilbert spaces
$L^2(\Omega)$ and $H^1_0(\Omega)$ with their respective
topological dual spaces.

However, we may substitute the identification of
$H^1_0(\Omega) \cong H^{-1}(\Omega)$ with the multidimensional
integration by parts formula that arises from the divergence theorem
and otherwise follow the proof of \cite[Thm 8.60]{leoni2017first}
using $L^2(\Omega)$ instead of $V^*$.
This approach allows us to use only the continuous embeddings
$V \hookrightarrow H^1_0(\Omega) \hookrightarrow L^2(\Omega)$.
For completeness, we sketch the modified proof below.
Note that the assumed boundary regularity that $\Omega$
is convex or of class $C^2$ (see \cref{sec:guiding_example})
is sufficient for this argument.

\begin{proposition}\label{prp:embedding}
Consider $\mathcal{H} = \left\{ u \in L^2(0,T; V)\,\vert\,\partial_t u \in L^2(0,T; L^2(\Omega))\right\}$ with
$\|u\|_{\mathcal{H}} = \|u\|_{L^2(0,T; V)} + \|\partial_t u\|_{L^2(0,T; L^2(\Omega))}$ for $u \in \mathcal{H}$.
Then the continuous embedding $\mathcal{H} \hookrightarrow C([0,T], H^1_0(\Omega))$ holds because there exists $C > 0$ such that
\[ \sup_{t \in [0,T]} \|u(t)\|_{H^1_0(\Omega)} \le
   C\left(\|u\|_{L^2(0,T; V)} + \|\partial_t u\|_{L^2(0,T; L^2(\Omega))}\right)
\]
holds for all $u \in \mathcal{H}$.
\end{proposition}
\begin{proof}
Let $u \in \mathcal{H}$. We use extension by reflection
to extend the function $u$ to the interval $(-\beta,T + \beta)$ for
some $\beta > 0$. We smooth $u$ with a family of standard
mollifiers $(\varphi_{\varepsilon})_{\varepsilon > 0}$ that are compactly 
supported in $(-\beta,T + \beta)$ and define
$u_\varepsilon \coloneqq u * \varphi_\varepsilon$.
Then we obtain $u_\varepsilon \to u \in L^2((0,T), V)$ and
$\partial_t u_\varepsilon \to \partial_t u \in L^2((0,T), L^2(\Omega))$.
We highlight that for the convergence
$\partial_t u_\varepsilon \to \partial_t u$ it is important
that the mollification of the derivative is the derivative of
the mollification. A cutoff argument to prove this  works  only
by virtue of the extension to the interval $(-\beta,T + \beta)$,
and we cannot extend it using absolute 
continuity because this is essentially what is to be shown.

Now, the mollification gives that $u_\varepsilon$, $\partial_t u_\varepsilon \in C_c^\infty(\R,V)$
and thus $u_\varepsilon$, $\partial_t u_\varepsilon \in C_c^\infty(\R,H^1_0(\Omega))$.
As in \cite[(8.32)]{leoni2017first}, we obtain for $x$, $x_0 \in [0,T]$
that
\[ \|u_\varepsilon(x)\|_{H^1_0(\Omega)}^2
   = \|u_\varepsilon(x_0)\|_{H^1_0(\Omega)}^2
     + 2 \int_{x_0}^x (\partial_t u_\varepsilon(s), u_\varepsilon(s))_{H^1_0(\Omega)}\,\dd s, \]
where $(\cdot,\cdot)_{H^1_0(\Omega)}$ is the usual inner product
on $H^1_0(\Omega)$. In particular, we can write
\[ \|u_\varepsilon(x)\|_{H^1_0(\Omega)}^2
   = \|u_\varepsilon(x_0)\|_{H^1_0(\Omega)}^2
     + 2 \int_{x_0}^x \int_\Omega \nabla \partial_t u_\varepsilon(s)^T \nabla u_\varepsilon(s)\dd\omega\dd s,\]
which allows us to apply multidimensional integration by parts that
follows from the divergence theorem to deduce 
\[
\|u_\varepsilon(x)\|_{H^1_0(\Omega)}^2
   = \|u_\varepsilon(x_0)\|_{H^1_0(\Omega)}^2
     + 2 \int_{x_0}^x 
     \int_{\partial\Omega} \partial_t u_\varepsilon(s) \nabla u_\varepsilon(s)\cdot\dd \sigma - 
     \int_\Omega \partial_t u_\varepsilon(s) \Delta u_\varepsilon(s)\dd\omega\dd s. \]
Since $\partial_t u_\varepsilon(s) \in H^1_0(\Omega)$ for all
$s \in [x_0,x]$ by virtue of the mollification, we obtain
\[ \|u_\varepsilon(x)\|_{H^1_0(\Omega)}^2
   = \|u_\varepsilon(x_0)\|_{H^1_0(\Omega)}^2
     - 2 \int_{x_0}^x 
     \int_\Omega \partial_t u_\varepsilon(s) \Delta u_\varepsilon(s)\,\dd\omega\,\dd s, \]
which implies 
\[ \|u_\varepsilon(x)\|_{H^1_0(\Omega)}^2
   = \|u_\varepsilon(x_0)\|_{H^1_0(\Omega)}^2
     + 2 \|\partial_t u_\varepsilon\|_{L^2(0,T;L^2(\Omega))}^2
         \|u_\varepsilon\|_{L^2(0,T;V)}^2 \]
by virtue of H\"{o}lder's inequality and $V\hookrightarrow L^2(\Omega)$.

Now the remainder of the proof of \cite[Thm 8.60]{leoni2017first}
applies if the dual space of $V$ is replaced with $L^2(\Omega)$
and the duality pairing $\langle \partial_t u(s), u(s) \rangle_{V^*,V}$
is replaced with $\int_\Omega \partial_t u(s) \Delta u(s)\dd\omega$.
\end{proof}

\begin{remark}
The fact that $\partial_t u_\varepsilon(s)$ is in 
$H^1_0(\Omega)$, which follows from $u(s) \in H^1_0(\Omega)$,
seems to be crucial for the proof of \cref{prp:embedding}.
However, there is also an abstract argument based on interpolation spaces.
In particular, one may combine \cite[Thm 4.10.2]{amann1995linear}
(choices $E_0 = H^2(\Omega)$, $E_1 = L^2(\Omega)$, $p = 2$) with the
continuous embedding
$H^1_0(\Omega) \hookrightarrow H^1(\Omega)$
to obtain that $\mathcal{H} \hookrightarrow C([0,T], H^1_0(\Omega))$,
where the fact that $u(s) \in H^1_0(\Omega)$ seems to be irrelevant.
\end{remark}

\begin{acronym}[Bash]
 \acro{FEM}{finite-element method}
 \acro{IBVP}{initial boundary value problem}
 \acro{MIOCP}{Mixed-Integer Optimal Control Problem}
 \acro{MPC}{model predictive control}
 \acro{OCP}{optimal control problem} 
 \acro{PCA}{principal component analysis}  
 \acro{PDE}{partial differential equation} 
 \acro{POD}{Proper Orthogonal Decomposition} 
 \acro{SVD}{singular-value decomposition}  
\end{acronym}

\bibliographystyle{spmpsci}
\bibliography{refs}

\bigskip
\framebox{\parbox{.92\linewidth}{The submitted manuscript has been created by
UChicago Argonne, LLC, Operator of Argonne National Laboratory (``Argonne'').
Argonne, a U.S.\ Department of Energy Office of Science laboratory, is operated
under Contract No.\ DE-AC02-06CH11357.  The U.S.\ Government retains for itself,
and others acting on its behalf, a paid-up nonexclusive, irrevocable worldwide
license in said article to reproduce, prepare derivative works, distribute
copies to the public, and perform publicly and display publicly, by or on
behalf of the Government.  The Department of Energy will provide public access
to these results of federally sponsored research in accordance with the DOE
Public Access Plan \url{http://energy.gov/downloads/doe-public-access-plan}.}}
\end{document}